\documentclass[12pt,reqno]{amsart}

\usepackage{amsmath}
\usepackage{amsthm}

\usepackage{amssymb}
\usepackage{graphicx, color}
\usepackage{import}
\usepackage[all]{xy}
\usepackage{graphicx} % Required for inserting images
\usepackage{subcaption}
\usepackage{hyperref}
\usepackage{cleveref}
\usepackage[perc]{overpic}
\usepackage{enumerate}
\usepackage{todonotes}
\usepackage[normalem]{ulem}
\usepackage{comment}

\usepackage{enumitem}
\setlist{itemsep=7pt,topsep=7pt}

\makeatletter
\def\thm@space@setup{%
  \thm@preskip=12pt plus 4pt minus 6pt
  \thm@postskip=15pt
}
\makeatother

\usepackage[margin=1.2in]{geometry}

\AtBeginDocument{%
   \def\MR#1{}
}

\theoremstyle{plain}
\newtheorem{theorem}{Theorem}[section]
\newtheorem*{theorem*}{Theorem}
\newtheorem{proposition}[theorem]{Proposition}
\newtheorem{lemma}[theorem]{Lemma}
\newtheorem{corollary}[theorem]{Corollary}

\newtheorem{conjecture}[theorem]{Conjecture}

\newtheorem{assumption}[theorem]{Assumption}

\newtheorem*{namedtheorem}{\theoremname}
\newcommand{\theoremname}{testing}
\newenvironment{named}[1]{\renewcommand{\theoremname}{#1}\begin{namedtheorem}}{\end{namedtheorem}}

\theoremstyle{definition}
\newtheorem{definition}[theorem]{Definition}
\newtheorem{question}[theorem]{Question}
\newtheorem{example}[theorem]{Example}
\newtheorem{remark}[theorem]{Remark}

\newtheorem*{fact*}{Fact}

\def\reals {\hbox {\rm {R \kern -2.8ex I}\kern 1.15ex}}
\def\integers {\hbox {\rm { Z \kern -2.8ex Z}\kern 1.15ex}}
\def\naturals {\hbox {\rm {N \kern -2.8ex I}\kern 1.20ex}}
\def\rationals {\hbox {\rm { Q \kern -2.2ex l}\kern 1.15ex}}
\def\hyp {\hbox {\rm {H \kern -2.7ex I}\kern 1.25ex}}

\def\from{\thinspace\colon}

\newcommand{\NN}{\mathcal{N}}
\newcommand{\calS}{\mathcal{S}}

\newcommand{\bdy}{\partial}

\newcommand{\ssm}{\smallsetminus}
\newcommand{\bea}{\begin{eqnarray*}}
\newcommand{\eea}{\end{eqnarray*}}

\newcommand{\abs}[1]{{\left\vert\, #1 \,\right\vert}}

\DeclareMathOperator{\id}{id}

\catcode `\@=11
\long\def\@savemarbox#1#2{\global\setbox#1\vtop{\hsize\marginparwidth 
\@parboxrestore\tiny\raggedright #2}}
\def\strutdepth{\dp\strutbox}
\def \ss{\strut\vadjust{\kern-\strutdepth \sss}}
\def \sss{\vtop to \strutdepth{
\baselineskip\strutdepth\vss\llap{$\diamondsuit\;\;$}\null}}

\def\strutdepth{\dp\strutbox}
\def \sst{\strut\vadjust{\kern-\strutdepth \ssss}}
\def \ssss{\vtop to \strutdepth{
\baselineskip\strutdepth\vss\llap{$\spadesuit\;\;$}\null}}

\def\strutdepth{\dp\strutbox}
\def \ssh{\strut\vadjust{\kern-\strutdepth  \sssh}}
\def  \sssh{\vtop to \strutdepth{
\baselineskip\strutdepth\vss\llap{$\heartsuit\;\;$}\null}}

\title{Rigidity of highly twisted plat diagrams}
\begin{document}

\author{Nir Lazarovich}

\author{Yoav Moriah}

\author{Tali Pinsky}

\author{Jessica S.\ Purcell}

\thanks{The authors would like to thank the Technion, Haifa for its
  support, and the Institute for Advanced Study (IAS), Princeton. 
  Lazarovich was partially supported by the Israel Science 
  Foundation (grant no. 1576/23). Purcell
  was partially supported by the Australian Research Council.}

%\date{\today}

\subjclass[2020]{Primary 57K45}

\keywords{knot diagrams, plats, twist regions, vertical spheres, 
canonical projections}

 \address{Department of Mathematics\\
 Technion\\
 Haifa, 32000 Israel}
 \email{lazarovich@technion.ac.il}

 \address{Department of Mathematics\\
 Technion\\
 Haifa, 32000 Israel}
 \email{ymoriah@technion.ac.il}

 \address{Department of Mathematics\\
 Technion\\
 Haifa, 32000 Israel}
 \email{talipi@technion.ac.il}

 \address{School of Mathematics\\
   Monash University, VIC 3800 Australia}
 \email{jessica.purcell@monash.edu}

\begin{abstract}
  In this paper we prove that if a knot or link has a
  sufficiently complicated plat projection, then that plat projection
  is unique. More precisely, if a knot or link has a $2m$-plat projection,
  where $m$ is at least four, and height at least two, and each twist 
  region of the plat contains at 
  least  four crossings, then such a projection is unique up to obvious 
  rotations. In particular, this projection gives a canonical form for 
  such knots and links, and thus provides a classification of
  these links.
\end{abstract}

\maketitle

\section{Introduction}\label{sec:Introduction}

One way of studying knots in $S^3$ is via their regular projections on
$2$-spheres in $S^3$.  Such projections are called \emph{knot
 diagrams}. Deciding when two diagrams correspond to the same knot is
a difficult problem, going all the way back to work of Tait in the
1870s. In 1926, K.~Reidemeister proved that any two regular
projections are equivalent by a sequence of \emph{Reidemeister moves};
see~\cite{Rei}. However, determining when two diagrams are equivalent
by Reidemeister moves is also a very difficult problem, and remains an
area of active research, for example see \cite{lackenby:reidemeister,
coward-lackenby}.

The earliest attempts to classify knots began in the 1870's by Tait,
using the \emph{crossing number} of a knot as the classifying
parameter. Recently, knots with up to 20 crossings have been classified, 
by Burton~\cite{Burton:19} and Thistlethwaite~\cite{Thistlethwaite:20}. 
That is, some 150 years later, distinguishing diagrams using crossing number is
possible only for knots with a small number of crossings.

Another way around the problem of deciding when diagrams are
equivalent would be to obtain ``canonical'' projections for knots. A
first attempt to do so was due to Schubert in 1956 see ~\cite{Schu}. 
All knots have $2m$-plat projections for some $m \in \mathbb{N}$.  
Schubert successfully classified all $4$-plats, which are more commonly known
as $2$-bridge knots and links. He showed that these
knots are classified by a pair of rational numbers, and that any continued
fraction expansion of these number correspond to a $4$-plat
projection. For a discussion of this see~\cite{BleilerMoriah}.

In a somewhat different flavor Menasco and Thistlethwaite proved
in 1993 that any two alternating projections of a knot
$\mathcal{K}\subset~S^3$ are equivalent by flype
moves~\cite{MT}. Flypes involve the existence of 4-punctured
spheres in the diagram, and they can be detected in a finite number of
steps. However, the sequence of such steps can be arbitrarily long.

In this paper, we prove a uniqueness statement for diagrams of an
infinite class of plats. We show that such plats have a canonical 
form which is unique, and can be read off the diagram immediately, 
without any need to consider equivalence relations such as Reidemeister 
moves or flypes. 

\begin{theorem}\label{thm:UniquePlat}
Let $K'$ and $K$ be two $4$-highly twisted plats 
representing the same knot or link $\mathcal{K} \subset S^3$, so that 
each diagram has width greater than or equal to $4$, and odd height greater 
than or  equal to $3$. Then $K=K'$ up to rotation in a 
vertical and/or a horizontal axis of the plats.
\end{theorem}

Here, when we say two diagrams $K'$ and $K$ represent the same knot or link, 
we mean that there exists an ambient isotopy $\varphi\from (S^3,K') \to (S^3,K)$.

For precise definitions of \emph{plats}, and \emph{$4$-highly twisted}, 
see Section~\ref{sec:Preliminaries} and references~\cite{BuZi} and~\cite{JM}.  

\Cref{thm:UniquePlat} shows that for knots and links satisfying the required
conditions there is a canonical form, namely the $2m$-plat projection. This 
canonical form has other nice features, as it gives
information about incompressible surfaces in $S^3 \ssm \NN(\mathcal{K})$;
see~\cite{FM1},~\cite{FM2} and~\cite{Wu}. In many cases, it gives information 
about the fundamental group of the knot space $\pi_1(S^3 \ssm \NN(\mathcal{K}))$ 
and its rank, that is the minimal
number of generators of $\pi_1(S^3 \ssm \NN(\mathcal{K}))$; see~\cite{LM}. 
It also gives information about the bridge number for $\mathcal{K}$, 
about Heegaard splittings of $S^3 \ssm \NN(\mathcal{K})$ and about manifolds 
obtained by Dehn surgery on $K$; see ~\cite{LM} and ~\cite{Wu}. 

In addition, a knot or link $\mathcal{K}$ with at least $C$ crossings in each 
twist region, for appropriate $C$, is known to satisfy several nice 
geometric properties, regardless of having the plat projection required by
Theorem~\ref{thm:UniquePlat}. For example, if $C=3$, then $K$ is hyperbolic 
~\cite{LMP,futer-purcell}. If $C=6$, then all Dehn fillings of $K$ are 
hyperbolic~\cite{futer-purcell}. 
Closed embedded essential surfaces are known to be high genus~\cite{BFT}. 
If $C=7$, there are known explicit upper and lower bounds on the hyperbolic 
volume of $K$~\cite{FKP}. And if $C=116$, then the shape of the cusp 
of $\mathcal{K}$ is bounded~\cite{purcell:cusps}. 

A $2m$-plat diagram comes with a family of \emph{horizontal bridge spheres}, defined 
in \Cref{sec:Preliminaries}. Such a sphere separates the diagram into two trivial 
tangles. The braids associated with such trivial tangles were investigated by Hilden 
in 1975~\cite{Hilden}. Our results have applications to Hilden double cosets as in 
the following corollary. 

\begin{named}{\Cref{cor:HildenDoubleCosets}}
If $b$ and $b'$ are two $4$-highly twisted words in $\mathcal{X}_{2m}$ with the same 
Hilden double coset $\mathcal{H}b\mathcal{H} = \mathcal{H}b'\mathcal{H}$, then $b$ 
and $b'$ are the same up to a vertical and/or a horizontal rotation of angle $\pi$. 
% In particular, as 
% elements in $\mathcal{B}_{2m}$ we have $b'=b$ or $b' = d b d^{-1}$,  where $d$ is the 
% fundamental braid.
\end{named}

We suspect that \Cref{thm:UniquePlat} can be improved, as follows.
\begin{conjecture}\label{conj: three also} The statement of Theorem 
\ref{thm:UniquePlat} holds for knots and links which have plat presentations
of width greater than or equal to $3$ and are $3$-highly twisted.
\end{conjecture}

It is also likely that the requirement that \emph{all} the coefficients are
$3$-highly twisted can be weakened as well. 
For example, the proof of \Cref{thm:UniquePlat} uses a result on essential 
surfaces originally proved for highly twisted plats by Finkelstein and 
Moriah~\cite{FM1, FM2}, which was later shown by Wu ~\cite{Wu} to remain 
true with fewer crossings in all but the outermost twist regions. We suspect 
\Cref{thm:UniquePlat} may generalise in this way as well. 

Going back to the question of crossing number, we ask the following:

\begin{question}\label{con:CrossingNumber}
Do knot and link diagrams satisfying the conditions of
Theorem~\ref{thm:UniquePlat} realize the crossing number of the
associated links?
\end{question}

For $2$-bridge knots the answer is negative since a 4-highly twisted diagram is not 
necessarily alternating and the crossing number is achieved by the alternating diagram 
\cite{kauffman1987state, murasugi1987jones,thistlethwaite1988kauffman}.

\subsection{Organisation}
In \Cref{sec:Preliminaries}, we review definitions of plats. In 
\Cref{Sec:VertSpheres}, we discuss a family of surfaces embedded in such plats, known 
to be essential. Section~\ref{sec:superincompespheres} puts these surfaces into 
standard position. Consequences are obtained in \Cref{sec:NoTwistRegions}. In 
\Cref{sec:Collections}, we use these surfaces and their images under isotopy to show 
that twist regions agree in two highly twisted plat diagrams. In \Cref{sec:2bridge}, 
we extend to twist regions on either end of the diagram. In Section~\ref{sec:Applications} 
we conclude the proof of Theorem ~\ref{thm:UniquePlat} and give applications to braids. 

\section{Highly twisted plat projections}\label{sec:Preliminaries}

Let $\mathcal{B}_{2m}$ denote the braid group on $2m$ strands and let 
$\sigma_1,\dots,\sigma_{2m-1}$ denote its $2m - 1$ standard 
generators. A word in the alphabet 
$\mathcal{X}_m = \{\sigma_1^{\pm 1},\dots,\sigma_{2m-1}^{\pm 1}\}$
determines a braid diagram with $2m$ strands (drawn from top to bottom) 
on a plane $P$, and an element in $\mathcal{B}_{2m}$.

\begin{definition}\label{def: highly twisted braid}
A word $b$ in $\mathcal{X}_m$ is in \emph{standard form of width $m$ and 
height $n$} if it is written as the concatenation of sub-words
$b_1\cdot b_2 \cdot \dots \cdot b_{n}$, 
where each $b_i$ has the following form:
\begin{enumerate}
\item When $i$ is odd, $b_i$ is a product of all $\sigma_j$ with $j$
  even. Namely:
\[b_i = \sigma_2^{-a_{i, 1}} \cdot \sigma_4^{-a_{i, 2}} \cdot \dots \cdot
\sigma_{2m-2}^{-a_{i, m-1}}\]

\item When $i$ is even, $b_i$ is a product of all $\sigma_j$ with $j$
  odd. Namely:
\[b_i = \sigma_1^{-a_{i, 1}} \cdot \sigma_3^{-a_{i, 2}} \cdot \dots \cdot
\sigma_{2m - 1}^{-a_{i, m}}\]
\end{enumerate}

For $c\ge 0$, the word $b$ in standard form is \emph{$c$-highly twisted} if 
$\abs{a_{i,j}}\geq c$ for all $i,j$.
\end{definition}

\begin{remark}
Note that the common conventions for the sign of twists in knot diagrams 
and braids differ. We prefer to work with the knot diagram convention. 
This is the reason for the negative exponents in \Cref{def: highly twisted braid}.
\end{remark}

\begin{example}\label{plat example}
Let $b$ be the the word in standard form of width $m=4$ and height $n=3$
and parameters 
\[ 
a_{i,j} = \begin{cases} 6 &(i,j)=(2,2)\\ -6 &(i,j)=(3,3)\\ -4 &\text{otherwise}\end{cases}
\]
Then the corresponding $4$-highly twisted braid diagram is shown in black in \Cref{fig:plat}. 
\end{example}

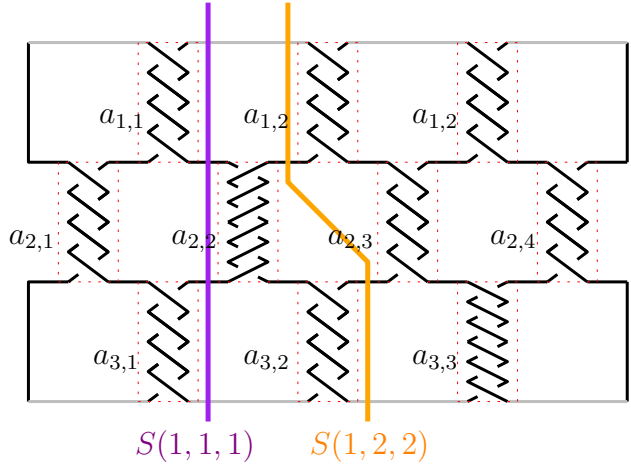
\begin{figure}
 \begin{overpic}[width=8cm]{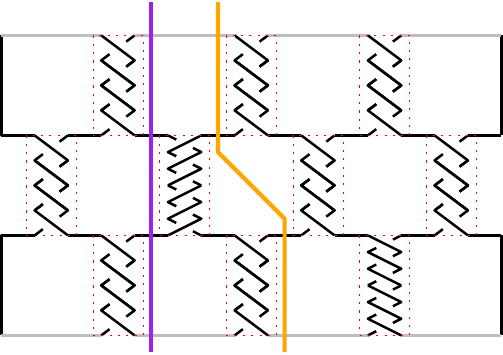}
        \put(12,50){$a_{1,1}$}
        \put(36,50){$a_{1,2}$}
        \put(64,50){$a_{1,2}$}
        \put(-3,30){$a_{2,1}$}
        \put(24,30){$a_{2,2}$}
        \put(50,30){$a_{2,3}$}
        \put(77,30){$a_{2,4}$}
        \put(11,10){$a_{3,1}$}
        \put(36,10){$a_{3,2}$}
        \put(64,10){$a_{3,3}$}
        \put(18,-5){\color{violet}$S(1,1,1)$}
        \put(47,-5){\color{orange}$S(1,2,2)$}
    \end{overpic}
    \bigskip

\caption{An 8-plat projection of a $4$-bridge link, and two vertical spheres.}
\label{fig:plat}
\end{figure}

\begin{definition}\label{def: plat closure}
Let $b$ be a word in $\mathcal{X}_m$. 
Let $x_1,\dots,x_{2m}$ (resp.\ $y_1,\dots,y_{2m}$) be the top (resp.\ bottom) 
endpoints of the strands in the braid diagram determined by $b$, ordered from left to 
right. Then, the \emph{plat closure of $b$} is a knot or link diagram obtained by 
connecting the pairs $\{x_1,x_2\},\dots,\{x_{2m-1},x_{2m}\}$ and the pairs 
$\{y_1,y_2\},\dots,\{y_{2m-1},y_{2m}\}$ by small unknotted disjoint arcs in the plane 
of projection $P$.
\end{definition}

In \Cref{fig:plat}, the arcs we add in order to form the plat closure 
are depicted in gray.

\begin{remark}
If two words describe the same braid in $\mathcal{B}_{2m}$, 
then their plat closures are diagrams of the same knot or link.
\end{remark}

It is well-known that every knot or link $\mathcal{K}$ has a diagram which 
is the plat closure of some word $b$ in $\mathcal{X}_m$ for some $m$, 
see \cite[p.~24]{BuZi}. Such a diagram is called a \emph{($2m$-)plat projection} 
of $\mathcal{K}$. The minimal number $m$ is called the \emph{bridge number} of 
$\mathcal{K}$.
 
Recall that an $m$-bridge sphere of a knot or link  $\mathcal{K}\subset~S^3$ 
is a $2$-sphere which meets $\mathcal{K}$ in $2m$ points and cuts
$(S^3,\mathcal{K})$ into two $2m$-string trivial tangles. For a $2m$-plat
projection, there is a well-known \emph{horizontal} $m$-bridge sphere, 
as follows: Arrange the maximum and minimum points of the projection 
to lie on horizontal lines at the top and bottom of the diagram, as in
Figure~\ref{fig:plat}. Any horizontal line in the projection plane lying 
below the maximum points, meeting the diagram in exactly $2m$ points, 
defines a bridge sphere: Connect the endpoints of the line by a  simple 
arc on the projection plane that does not meet $K$ to form a closed curve. 
Capping this curve by two disks, one above and one below the plane $P$, 
results in a  bridge sphere $\Sigma$. The key point is that a horizontal 
sphere obtained in this manner separates $K$ into two trivial tangles. 

\begin{definition}\label{def: highly twisted plat}
A plat projection  $K$ of a knot or a link $\mathcal{K}$ is in 
\emph{standard form of width $m$ and height $n$} 
if it is the plat closure of a word $b$ in standard 
form of width $m$ and height $n$ where $n$ is odd.
It is \emph{$c$-highly twisted} if $b$ is $c$-highly twisted.
\end{definition}

\begin{remark}\label{rem: even plats}
Note that a plat projection in standard form has height that is an odd number.
Analogues of this definition for even heights are given in 
\Cref{def: even and doubly even}.
\end{remark}

We remark that knots can be presented in  many different ways as plat 
projections, and that braids can be
presented in many different ways as words in $\mathcal{X}_m$ 
in standard forms. However, the main results of this paper show that under the 
additional assumption of being $4$-highly twisted these are unique, up to some 
obvious symmetries. Of course, not all knots and links admit $4$-highly twisted 
plat diagrams. Indeed,  any
such knot or link must be hyperbolic~\cite[Theorem~A]{LMP}. 

A \emph{twist region} in a knot diagram is a disk on the projection plane that 
contains a maximal chain of bigons describing a trivial integer 2-tangle.
In Figure~\ref{fig:plat} the twist regions are shown as red dotted rectangles, 
each containing the part of the diagram corresponding to the sub-word 
$\sigma_k^{-a_{i,j}}$ of $b$. The number $a_{i,j}$ is the (signed) number of crossings 
in a twist region. 

\section{Vertical spheres}\label{Sec:VertSpheres}

In this section we recall properties of a family of surfaces that 
were first defined in~\cite{FM2}; see also~\cite{Wu}. 

\begin{definition}\label{def:vertical-two-sphere} 
Let $K$ be a plat projection of a knot or link $\mathcal{K}\subset S^3$ in 
standard form of width $m$ and height $n$, where $m\geq 3$ and $n$ is odd.  
Let
$\alpha=\alpha(c_1,\dots,c_{n})$ be an arc running monotonically
from the top of the plat to the bottom so that $\alpha$ is disjoint 
from all twist regions, and there are $c_i$ twist regions to the left 
of the arc at the $i$-th row. We further require that there is at least 
one twist region on each side of $\alpha$ at each level 
(hence the requirement that $m\geq 3$), and that $\alpha$
intersects $K$ in precisely $n+1$ points. Now connect the endpoints of 
$\alpha$ by a simple arc $\beta$ lying in the projection plane $P$ 
and disjoint from $K$, 
and cap the simple closed curve $\gamma=\alpha\cup\beta$ by two disks,
one above $P$ and the other below $P$, to obtain a $2$-sphere
$S=S(c_1,\dots,c_{n})$. Any sphere isotopic to $S$ in $S^3\ssm K$ is called a
\emph{vertical $2$-sphere} and is also denoted by $S(c_1,\dots,c_{n})$. 
We always assume that the projection
plane and vertical spheres intersect transversally.

\end{definition}

In \Cref{fig:plat} the intersection of the vertical sphere 
$S(1,1,1)$ with the projection plane $P$ is shown in purple, and the 
intersection of the sphere $S(1,2,2)$ is shown in orange.

Wu proved that vertical 2-spheres are essential under 
mild conditions on the diagram~\cite{Wu}, extending work of Finkelstein and
Moriah~\cite{FM1}. We will use these results here. We also need the 
fact that vertical 2-spheres are pairwise incompressible, as follows. 

\begin{definition}
A surface $S\subset (S^3,K)$ is \emph{pairwise incompressible} if, 
for any disk $D$ with $D\cap S=\bdy D$ that meets $K$ in a single 
point, the curve $\bdy D\subset S$ also bounds a disk in $S$ meeting 
$K$ in one point. Otherwise, it is \emph{pairwise compressible} and 
$D\ssm K$ is called a \emph{pairwise compressing annulus}. 

A sphere is \emph{super-incompressible} if it is incompressible, 
boundary incompressible, and pairwise incompressible.
\end{definition}

\begin{theorem}\label{Thm:super incompressible}
Suppose a knot or link $\mathcal{K}\subset S^3$ has a $2m$-plat  projection 
that is $3$-highly twisted, and $m\geq 3$. Then any vertical 2-sphere 
$S=S(c_1, \dots, c_n)$ is super-incompressible in $S^3\ssm \NN(\mathcal{K})$. 
\end{theorem}

\begin{proof}
Wu showed that vertical spheres in plat projections are incompressible 
and boundary incompressible~\cite[Theorem~1]{Wu}. We give a new proof 
which shows pairwise incompressibility using~\cite{LMP}. 

Consider the vertical sphere $S=S(c_1,\dots,c_n)$, and assume in contradiction 
that $S$ has an essential compression disk, boundary compression disk, or 
pairwise compressing annulus, which we denote by $E$. Without loss of 
generality, assume $E$ lies on the left of $S$. Consider the tangle 
diagram to the left of $S$. Double this diagram along $S$ to produce 
a new $3$-highly twisted diagram. In the double, $E$ also doubles to 
become an essential sphere, disk, or annulus $D(E)$ embedded in the 
complement of the doubled diagram. On the other hand the double diagram 
can be turned twist-reduced by performing some obvious flypes: More precisely, 
it is twist-reduced unless $c_i=1$ for some even $i$;
in this case, perform flypes to obtain a twist reduced diagram. 
The flypes only merge or remove existing twist regions, hence the 
result remains $3$-highly twisted. Finally, the result of the flypes 
is prime, because the original diagram is prime and there is no 
nontrivial arc on the left of $S$ meeting the diagram at a single point.  

So the double has a prime, twist-reduced $3$-highly twisted diagram. 
By~\cite[Theorem~A]{LMP} it follows that the complement of the link described 
by the doubled diagram is hyperbolic. But a hyperbolic link cannot admit an 
essential sphere, disk, or annulus, contradicting the existence of $D(E)$. 
\end{proof}

\begin{theorem}\label{thm:Verticalspheres}
Let $K$ be a $2m$-plat corresponding to a knot or non-split link
$\mathcal{K}\subset S^3$. Assume $K$ has width $m\geq 3$ and height $n$. 
Let $S$ be an embedded $2$-sphere in the pair $(S^3,\mathcal{K})$ 
so that $S \ssm \NN(\mathcal{K})$ is super-incompressible, and $S$ meets 
the knot $\mathcal{K}$ at most $n+1$ times. If $S$ does not pass through twist 
regions then $S$ is isotopic to a vertical sphere. 

More generally, if $\mathcal{S}= \bigcup_{i=1}^k S_i$ is a disjoint 
union of such spheres, then $\mathcal{S}$ is isotopic to a disjoint 
collection of vertical 2-spheres. 
\end{theorem}

\begin{proof} 
Let $S \subset (S^3, \mathcal{K})$ be a $2$-sphere satisfying the hypotheses of the 
theorem. Let $P$ denote the projection plane corresponding to the plat. Then 
$S\cap P \neq \emptyset$, or else $S$ would bound a 3-ball in $S^3\ssm \NN(K)$,
contradicting the assumption that $S$ is incompressible. 
Assume that $S$ intersects $P$ transversely. Moreover, we may assume that 
each component of the collection of simple closed curves $S\cap P$ must intersect 
$K$, for otherwise, since $S$ does not pass through a twist region, such a curve 
would bound a disk on $P$. By the incompressibility of $S\ssm \NN(K)$, 
such a disk would also bound a disk on $S$. The union of the two disks would 
form a $2$-sphere which bounds a 3-ball as $S^3\ssm \NN(K)$ is irreducible.
This $3$-ball can be used to isotope $S$ through $P$, eliminating the curve 
of intersection. 

Suppose $S\cap P$ contains at least two components; denote two of them by
$\gamma_1$ and $\gamma_2$. These components contain intersection points 
of $S \cap K$. Let $\delta$ be a simple closed curve on $S$ disjoint from 
$P$ separating $\gamma_1$ and $\gamma_2$. Hence $\delta$ is contained in 
one of the two $3$-balls $B_1$ and $B_2$ constituting $S^3 \ssm P$, say 
$B_1$. As such it bounds a disk $\Delta \subset B_1$. Since the curve 
$\delta$ separates the intersection points of $S \cap K$ which are contained 
in $\gamma_1$ from those contained in $\gamma_2$, it is an essential 
curve in $S \ssm \NN(K)$, and $\Delta$ gives rise to a compressing disk 
for $S \ssm \NN(K)$, in contradiction. Hence $S\cap P$ contains a single 
component.

We now show that the component $S\cap P$ defines a vertical 2-sphere as in
\Cref{def:vertical-two-sphere}. Note it must run from the top of the diagram 
to the bottom, else there is a horizontal bridge sphere disjoint from $S\cap P$
that splits the link complement into two trivial tangles. However, there are no
super-incompressible surfaces in a trivial tangle, so this is impossible. 
We claim that at each level, there is at least one twist region on the inside 
and on the outside of $S\cap P$, for otherwise there would be an obvious
compression disk or pairwise compressing annulus. Since $S\cap P$ runs top to
bottom, with twist regions on each level on each side, it must meet $K$ at 
least $n+1$ times. Since by hypothesis it meets $K$ at most $n+1$ times, it must 
meet $K$ exactly $n+1$ times. It follows that up to isotopy the arc of $S\cap P$ 
running from top to bottom meeting $K$ must be monotonic. Then $S$ satisfies 
all the conditions in the definition of a vertical 2-sphere. 

Finally, note that all the above isotopies apply to isotope a disjoint 
union of 2-spheres to a disjoint union of vertical 2-spheres.
\end{proof}

\section{Isotoping super-incompressible spheres}\label{sec:superincompespheres}

In this section we isotope families of super-incompressible surfaces into 
a ``nice'' position with respect to a $2m$-plat diagram. 

\begin{definition}\label{def:P^_, P^+}
Let $\mathcal{K} \subset S^3$ be a knot with a $2m$-plat  projection $K$ on 
a plane $P$. Let $c$ denote a crossing in the knot diagram. A \emph{bubble} is 
a small $2$-sphere $S_c$ that intersects $P$ in an equator circle bounding a small 
disk $D_c \subset P$ that contains only the crossing $c$. The equator 
$\partial D_c$ divides  $S$ into two hemispheres $S_c^-$ and $S_c^+$.  
We may isotope $K$ to lie on $P$ except at the bubbles, where one arc 
of $K$ runs across the hemisphere $S_c^+$ and one runs across $S_c^-$. 

Let $P^-$ denote the plane obtained by removing the disks $D_c$ for all crossings
and attaching hemispheres $S^-_c$ along $\partial D_c$ instead. Similarly let $P^+$
denote the plane obtained from $P$ by attaching the hemispheres $S^+_c$ for
all crossings $c$. Both $P^+$ and $P^-$ bound $3$-balls in $S^3 \ssm K$.
\end{definition}

\begin{definition}\label{def:general position}
A surface $S \subset S^3 \ssm \NN(K)$, possibly with punctures corresponding to
meridional curves on $\NN(K)$, is said to be in \emph{general position} if:
\begin{enumerate}
\item $S$ intersects the planes  $P^\pm$ transversally.

\item For every twist region $T$, each component of the intersection of 
$S\cap T$ intersects $P^\pm$ as indicated in Figure~\ref{fig:general position},
left. That is, $S$ is disjoint from $K$ within each twist region, and forms a disk
between the two strands of $K$. 

\item In particular, for any crossing $c$, the bubble $S_c$ bounds a ball. The surface 
$S$ meets the ball bounded by $S_c$ only in \emph{saddles}, which are 
disks lying between the two strands of $K$, disjoint from $K$. See 
Figure \ref{fig:general position}, right.
\end{enumerate}
\end{definition}

\begin{figure}
    \centering
    \includegraphics[height = 3cm]{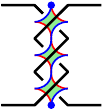}
    \hspace{3cm}
    \includegraphics[height=3cm]{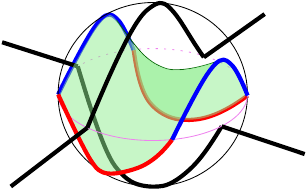}
    \caption{On the left, the intersection of a surface $S$ in general position 
    with a twist region. The intersection $S\cap P^+$ is shown in blue, and the 
    intersection $S\cap P^-$ is in red. On the right, a saddle in a bubble.}
    \label{fig:general position}
\end{figure}

The following lemma is proved in~\cite{LMP}, using methods similar to~\cite{MT}. 

\begin{lemma}
Suppose $K$ is has a 3-highly twisted plat projection. Then every 
surface $S \subset S^3 \ssm \NN(K)$ can be isotoped to be in general 
position.\qed
\end{lemma}

\begin{assumption}
    Up to isotopy, we may assume that the plat projection has the following form, 
shown in \Cref{fig:layers}: 
\begin{enumerate}
\item The strands of the link $K$ are either horizontal or vertical 
outside twist regions.
\item The vertical segments in $K$ appear only leftmost or rightmost segments.
\end{enumerate}
\end{assumption}

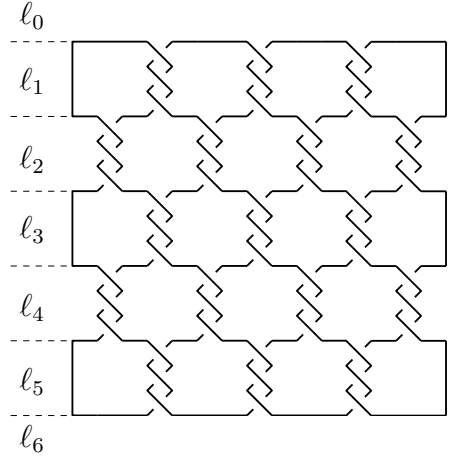
\begin{figure}[ht!]
    \centering
    \vspace{0.3cm}
    \begin{overpic}[height = 5cm]{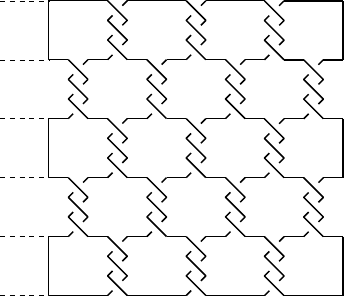}
        \put(2,90){$\ell_0$}
        \put(2,75){$\ell_1$}
        \put(2,57){$\ell_2$}
        \put(2,41){$\ell_3$}
        \put(2,23){$\ell_4$}
        \put(2,6){$\ell_5$}
        \put(2,-7){$\ell_6$}
    \end{overpic}
    \vspace{0.5cm}
   
    \caption{The decomposition of the projection plane into layers for an example with $n=5$.}
    \label{fig:layers}
\end{figure}

\begin{definition}
The projection plane is decomposed into $n+2$ horizontal regions called
\emph{layers}. The layer $\ell_0$ is a disk with an arc on its boundary 
meeting the top of the diagram of \Cref{fig:layers} (the full disk is 
not shown in that figure).  For each $i=1,\dots,n$, a layer $\ell_i$ is 
an annulus containing the $i$-th row of twist regions; one of its 
boundary components agrees with a boundary component of $\ell_{i-1}$.  
% If $n$ is odd,
The layer  $\ell_{n+1}$ is a disk containing the region below the 
plat projection, and also the point at infinity. 
\end{definition}

\begin{definition}
Let $\ell_i$ be a layer. 
The intersection $S \cap P^\pm \cap \ell_i$ consists of curves and arcs. A curve is 
called an \emph{O-curve} if $\ell_i$ is an annulus and the curve goes around the 
annulus once. An arc is called a \emph{U-arc} if its two endpoints are on the same 
boundary of $\ell_i$; otherwise it is called an \emph{I-arc}. See \Cref{fig:arcs} for 
an example. A U-arc is an \emph{extremal U-arc} if it contains a global maximum or 
minimum of the simple closed curve in $S\cap P^\pm$ containing it, with respect 
to the height function given by projecting to the $y$-coordinate in the plane $P$.

An endpoint of an arc of $S \cap P^\pm \cap \ell_i$ may lie on $K$, or may run to a 
bubble meeting both $\ell_i$ and $\ell_{i\pm 1}$; we call the latter an 
\emph{extremal bubble}. In either case, we call this endpoint an \emph{interface} 
point. If an endpoint is not an interface point, then it meets the boundary of $
\ell_i$ outside the plat, and will be called a \emph{free} point.

A point of $S\cap K$ is called an \emph{intersection point} of $S\cap K$. 
\end{definition}

\begin{figure}[ht]
\centering
\includegraphics{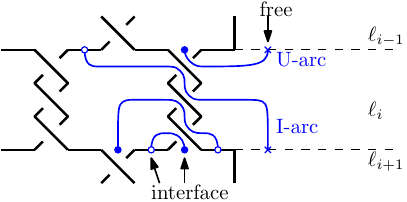}
\caption{An example of three U-arcs and one I-arc in $S\cap P^+ \cap \ell_i$. 
Free endpoints are marked with a cross, while interface endpoints are marked with 
round disks that are empty or full, corresponding to whether it is an intersection 
point or not, respectively.}
\label{fig:arcs}
\end{figure}

\begin{proposition}\label{removing unwanted intersections}
Let $K$ have a $2m$-plat, $m\geq 3$, of height $n$, 
and suppose further it is 3-highly twisted. Let $\calS$ be the union of a finite 
collection of disjoint super-incompressible spheres in $S^3\ssm \NN(K)$. Then 
up to isotopy, we may assume that for all $0\le i\le n+1$ and every component 
$S\subseteq  \calS$ the following hold:
\begin{enumerate}[label = ($\diamondsuit$\arabic*)]
\item\label{no vertical intersection}  $S\cap P^\pm$ does not intersect the 
vertical segments of $K$.

\item \label{menasco no scc} Every component of $S\cap P^\pm$ meets a twist region or 
an interface point. That is, there are no components that bound disks disjoint from 
$K$ in the diagram. 

\item\label{menasco assumptions} No component of $S\cap P^\pm$ visits a twist 
box more than once.

\item\label{no scc in layers} If a component of $S\cap P^\pm$ is contained in the 
layer $\ell_i$, then it is an O-curve (and $i$ is even).

\item\label{no trivial U-arcs} 
If a component of $S\cap P^\pm\cap \ell_i$, $1\le i\le n$, is a U-arc with at 
least one free endpoint, then it must pass through a 
twist region.

\item\label{no double intersection outside} 
No component of $S\cap P^\pm$ meets a layer in a pair of arcs that both have 
free endpoints. That is, for any curve $c$ of $S\cap P^\pm$, and among all 
arcs of $c\cap \ell_i$, at most one such arc has any free endpoint(s).
\end{enumerate}
\end{proposition}

\begin{proof}
Isotope $\calS$ to minimize the lexicographic complexity 
\begin{align*}
c(\calS)=(
&\#\{\text{intersections with vertical segments of }K\}, \\
&\#\{\text{free endpoints}\}, \\
&\#\{\text{saddles of $\calS$}\}, \\
&\#\{\text{pairs of arcs as in \ref{no double intersection outside}}\}, \\
&\#\{\text{components not satisfying \ref{menasco no scc}}\}
) 
\end{align*}

\ref{no vertical intersection}: If such an intersection exists, one can easily 
push $\calS$ upwards or downwards to remove it and reduce the complexity.

\ref{menasco no scc}: This is proved in Menasco \cite[Lemma 1]{menasco1984closed}; we 
repeat the proof here.
Let $c$ be an innermost curve in $\calS\cap P^\pm$ that does not meet a twist region 
or an intersection point. Let $D$ be the disk in $P^\pm$ with $\partial D = c$. Since 
$\calS$ is incompressible, there is a disk $D'$ in $\calS$ with $\partial D' = c$. 
Since $S^3 \ssm \NN(K)$ is irreducible, the disks $D\cup D'$ bound a 3-ball. Since $\calS$ is 
super-incompressible, no component of $\calS$ is contained in this 3-ball. Pushing 
$\calS$ through this 3-ball eliminates the curve $c$ and reduces the complexity. 

\ref{menasco assumptions}: If a component of $\calS\cap P^\pm$ visits a twist box more 
than once, then an innermost such component will meet a bubble more than once. By 
minimizing over the number of saddles, Menasco \cite[Lemma~1] {menasco1984closed} 
showed that no such intersections exist: If a component of $\calS\cap P^\pm$ meets 
a bubble twice, then there is an arc $\alpha$ of $\calS\cap P^\pm$ running away 
from the bubble, back to the bubble. If both its endpoints are on the same side of 
the knot in the bubble, the arc $\alpha$ cobounds a disk with an arc on the bubble. Push 
the disk slightly into the interior of the ball bounded by $P^\pm$, and use that disk 
to slide the arc of $\calS\cap P^\pm$ into the saddle and through. Menasco calls this 
a band move, and it reduces complexity; see 
also~\cite[Figure~14]{PurcellTsvietkovaI}. 
If it meets a bubble on the opposite side, then two innermost such arcs must be 
connected, and pairwise incompressible implies we may slide the surface off, again reducing 
number of saddles. We remark that both moves do not require the knot to be alternating.

\ref{no scc in layers}: Consider a component 
of $\calS\cap P^\pm$ contained in 
the layer $\ell_i$. By \ref{menasco assumptions}, it cannot meet a twist region since if 
it does it will have to meet it twice. It cannot have an intersection point with $K$, 
since those occur on the boundary of $\ell_i$. By \ref{menasco no scc}, it cannot bound a disk in a 
region.
Therefore, this component must be an $O$-curve.

\ref{no trivial U-arcs}: Suppose a U-arc has two free endpoints and does not meet a 
twist region. Then an innermost such arc bounds a disk with the 
boundary on $\ell_i$ that does not meet $\calS$ or $K$ in its interior. 
Pushing $\calS$ through this eliminates this curve, reducing the number of free 
endpoints, hence the complexity.

If a U-arc $\alpha \subset \calS\cap P^\pm \cap \ell_i$ has one free endpoint and does 
not meet a twist region, then necessarily $i$ is even. Assume $\alpha$ is innermost, 
and that the endpoints of $\alpha$ are on $\partial \ell_i \cap \partial \ell_{i+1}$. 
At the free endpoint of $\alpha$, it meets an arc $\beta$ in $\calS \cap P^\pm \cap \ell_{i+1}$.
The arc $\beta$ is situated between two vertical segments of $K$ and so both endpoints 
of $\beta$ are free. Thus by the previous paragraph it must be an I-arc. At its 
other endpoint, $\beta$ meets a curve $\gamma$ in $\ell_{i+2}$. Pushing $\calS$ as 
in \Cref{fig:removing uarcs} reduces the number of arcs.

\begin{figure}
    \centering
    \begin{overpic}[height=4cm]{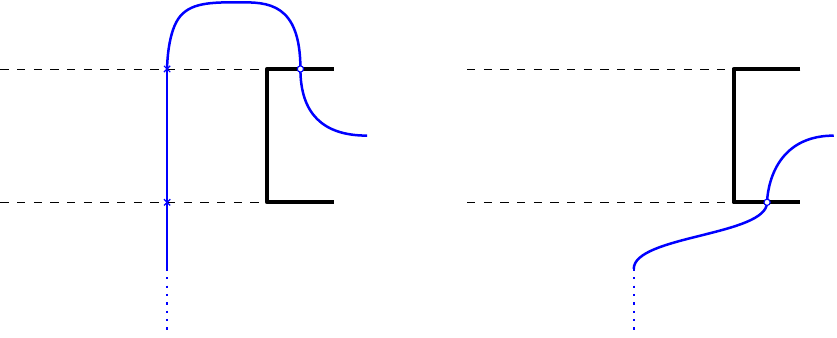}
        \put(14,35){$\color{blue}\alpha$}
        \put(14,22){$\color{blue}\beta$}
        \put(14,10){$\color{blue}\gamma$}
    \end{overpic}
    \caption{Removing U-arcs that do not pass through a twist region and 
    have one free endpoint.}
    \label{fig:removing uarcs}
\end{figure}

\ref{no double intersection outside}: Let $c$ be a curve of $\calS\cap P^\pm$ that meets the 
layer $\ell_i$ in a pair of arcs $\alpha_1$, $\alpha_2$ so that both have a free endpoint. 
Note that $\alpha_1$ and 
$\alpha_2$ could be I-arcs both on the same side of the plat (left or right), or I-arcs 
on opposite sides of the plat, with free endpoints either both at the top or both at the 
bottom, or one at the top and one at the bottom. Alternatively, $\alpha_1$ and $\alpha_2$ 
could be U-arcs.

In all cases, there exists an arc $\gamma$ in $\ell_i$, disjoint from $K$, such that 
the intersection $\gamma \cap c = \partial \gamma$ consists of the two free endpoints 
of $\alpha_1$ and $\alpha_2$. Note that $\gamma$ possibly meets pairs of other 
components of the collection $\calS$.

For an innermost subarc $\gamma'$ of $\gamma$ connecting two arcs belonging 
to the same component, there exists a disk $D$ that does not meet $\calS$ 
nor $K$ in its interior and  whose boundary is $\gamma'\cup\gamma''$ where 
$\gamma''$ is an arc on $\calS$. Push $\calS$ through $D$. This 
isotopy amounts to doing surgery to the component containing the pair of arcs 
along $\gamma$, splitting this component into two. Follow by an isotopy removing 
trivial components not meeting $K$ if such components have been created. This move 
either reduces or does not change the complexity.

After a finite number of successive such surgeries, $\gamma$ does not intersect 
$\calS$ in its interior. Proceed to perform surgery on $c$ along $\gamma$. After 
the surgery, the two arcs with free endpoints meeting $\gamma$ are eliminated, 
and two other arcs $\alpha_1',\alpha_2'$ appear. However, these arcs belong to 
different components. This reduces the number of pairs as in 
\ref{no double intersection outside}, and so reduces the complexity. 
\end{proof}

\section{Super-incompressible spheres and twist regions}
\label{sec:NoTwistRegions}

In this section we prove that super-incompressible spheres can always be 
isotoped to avoid twist regions. 

\begin{theorem}\label{thm: notwistregions}
Suppose $K$ is a $2m$-plat projection of a knot or a link of width 
$m\geq 4$ and height $n\geq 2$, and suppose the projection is $4$-highly twisted. 
Let $S$ be a super-incompressible sphere satisfying \ref{no vertical intersection}
-- \ref{no double intersection outside}, with at most  $n+1$ punctures. Then $S$ 
has exactly $n+1$ punctures and does  not pass through a twist region.
\end{theorem}

For the rest of this section, we will assume that $K$ and $S$ satisfy 
the hypotheses of \Cref{thm: notwistregions}.

\begin{definition}
For all $0\le i\le p$ define 
\begin{multline}\label{eq: definition of chi_i}
\chi_i =
\#\{\text{O-curves in }S\cap P^\pm \cap \ell_i\}- \#\{\text{internal saddles in }S\cap\ell_i\}\\
+ \tfrac12\#\{\text{extremal U-arcs in }S\cap P^\pm\cap\ell_i\}  - \tfrac12\{\text{interface 
points in }S\cap \bdy\ell_i\}
\end{multline}
\end{definition}

\begin{remark}
Note that the O-curves are counted separately in $P^+$ and $P^-$, and similarly for the 
extremal U-arcs. However, the number of internal saddles and interface points is 
independent of $P^\pm$. 
\end{remark}

\begin{lemma}\label{euler by layer}
The Euler characteristic of $S$ satisfies
\[ \chi(S)= \sum_{i=0}^{p}\chi_i \]
\end{lemma}

\begin{proof}
Consider the distance from the plane $P$ as a Morse function $h$ on the surface $S$. 
One can arrange it so that the minima of $h$ correspond to the disks in $S$ below 
$P^-$ bounded by simple closed curves in $S\cap P^-$. Similarly, its maxima 
correspond to simple closed curves in $S\cap P^+$. Its saddles correspond to 
the saddles in the bubbles. By Morse theory, 
\begin{multline}\label{eq: euler basic}
\chi(S) = \#\{\text{maxima}\} + \#\{\text{minima}\}  - \#\{\text{punctures}\}
- \#\{\text{saddle points}\} \\
=\#\{\text{simple closed curves in }S\cap P^\pm\} - \#\{\text{intersection points 
of }S\cap K\} \\-\#\{\text{saddles of }S\}
\end{multline}    

Each simple closed curve in $S\cap P^\pm$ is either contained in a layer, and is 
thus an O-curve, or obtains an extreme maximum and an extreme minimum with respect 
to the height on $P$ at two distinct layers, and thus 
\begin{multline}\label{eq: upper bound using minmax}
\#\{\text{simple closed curves in }S\cap P^\pm\} = \\
\sum_{i=0}^{p}
\left( \#\{\text{O-curves of }S\cap P^\pm \cap \ell_i\} + 
\tfrac12\#\{\text{extremal U-arcs of }S\cap P^\pm \cap \ell_i \} \right)
\end{multline}

Each internal saddle appears in one layer, and each interface point 
(i.e.\ external saddle or intersection point) is shared by two layers, and so
\begin{multline}\label{eq: upper bound of interface}
\#\{\text{intersection points of }S\cap K\} +\#\{\text{saddles of }S\} = \\
\sum_{i=0}^{p} \left(\#\{\text{internal saddles in }\ell_i\} + \tfrac12
\{\text{interface points on }\bdy\ell_i\}\right)
\end{multline}

The lemma now follows by combining \eqref{eq: euler basic},
\eqref{eq: upper bound using minmax} and \eqref{eq: upper bound of interface} and 
the definition of $\chi_i$.
\end{proof}

\begin{lemma}\label{Lem:Chi_i equivalence}
The value of $\chi_i$ is equivalent to the following:
\begin{multline}\label{eq: chi_i equivalent def}
\chi_i =
\#\{\text{O-curves in }S\cap P^\pm \cap \ell_i\}- \#\{\text{internal saddles in }S\cap\ell_i\}\\
- \tfrac12\#\{\text{non-extremal arcs in }S\cap P^\pm\cap\ell_i\}  + 
\tfrac12\{\text{free points in }S\cap \bdy\ell_i\}
\end{multline}
\end{lemma}

\begin{proof}
The total number of arcs of $\ell_i \cap P^{\pm} \cap S$ is equal to 
\begin{equation}\label{Eqn:number arcs}
\#\{\text{arcs in }\ell_i\cap P^{\pm}\} = \#\{\text{extremal arcs}\}+\#\{\text{non-extremal arcs}\}.
\end{equation}
Each arc has two endpoints.
The point at an endpoint of an arc is an endpoint of two arcs in $\ell_i$, one on $P^+$ and 
one on $P^-$. It is also either an interface endpoint or a free endpoint. Thus:
\begin{align*}
\#\{\text{arcs in }\ell_i\cap P^{\pm}\} &= 
\#\{p \in \bdy \ell_i \mid p \text{ an endpoint of an arc in }\ell_i\cap P^{\pm}\} \\
&= \#\{p \in \bdy\ell_i \mid p\text{ an interface point}\} + \#\{p \in \bdy\ell_i \mid p
\text{ a free point}\}.
\end{align*}
Combining the above with \cref{Eqn:number arcs} we get
\[\#\{\text{extremal arcs}\}+\#\{\text{non-extremal arcs}\} = 
\#\{\text{interface points}\} + \#\{\text{free points}\}\]
and the equivalence of \eqref{eq: definition of chi_i} and 
\eqref{eq: chi_i equivalent def} follows.
\end{proof}

\begin{definition}
Let $\alpha$ be a component of $S\cap P^\pm\cap \ell_i$. Define 
\begin{equation}
\chi_+(\alpha) = \begin{Bmatrix}
    1 & \text{if }\alpha\text{ is an O-curve}\\
    \tfrac12 & \text{if }\alpha\text{ is an extremal U-arc}\\
    0 & \text{otherwise}
\end{Bmatrix} - \tfrac14\#\{\text{interface endpoints of }\alpha\}\\
\end{equation}
\end{definition}

\begin{lemma}\label{Lem:ChiPlusFormula}
Equivalently,
\begin{equation}
    \chi_+(\alpha)=\begin{Bmatrix}
    1 & \text{if }\alpha\text{ is an O-curve}\\
    0 & \text{if }\alpha\text{ is an extremal U-arc}\\
    -\tfrac12 & \text{otherwise}
    \end{Bmatrix} + \tfrac14\#\{\text{free endpoints of }\alpha\}
\end{equation}
The only components $\alpha$ of $S\cap P^{\pm}\cap \ell_i$ with positive 
$\chi_+(\alpha)$ are O-curves or extremal U-arcs with at least one free endpoint.
\end{lemma}

\begin{proof}
The formulae are identical (and positive) for $\alpha$ an O-curve. 
If $\alpha$ is an arc, then
\[ 2 =\#\{\text{endpoints of }\alpha\} = \#\{\text{free endpoints of }\alpha\} + 
\#\{\text{interface endpoints of }\alpha\} \]
Equivalently, 
\[-\tfrac14 \#\{\text{interface endpoints of }\alpha\} = -\tfrac12 
+ \tfrac14\#\{\text{free endpoints of }\alpha\} \qedhere
\]
\end{proof}

\begin{lemma}\label{euler of layer is sum over arcs}
The value of $\chi_i$ satisfies
    \begin{equation}\label{eq: euler by arc}
        \chi_i = \sum_{\alpha} \chi_+(\alpha) 
        \; -\;\#\{\text{internal saddles in }\ell_i\}
    \end{equation}
    where the sum runs over every component $\alpha$ of  $S\cap P^\pm \cap \ell_i$.
\end{lemma}

\begin{proof}
Indeed, each interface point in $\ell_i$ is an endpoint of two arcs in $\ell_i$: one 
in $P^+$ and one in $P^-$. Equation \eqref{eq: euler by arc} is obtained from 
\eqref{eq: definition of chi_i} where instead of counting $(-\tfrac12)$ per 
interface point, we count $(-\tfrac14)$ per interface endpoint of an arc $\alpha$.
\end{proof}

\begin{lemma}\label{finding I arcs}\hfill
\begin{enumerate}[label=(\arabic*)]
\item\label{one twist region I-arcs} If $S \cap P^\pm \cap \ell_i$ meets the 
twist region $T$ but does not meet its neighboring twist regions, then there 
are at least two I-arcs through $T$. 
\item\label{two twist regions I-arcs} If $S \cap P^\pm \cap \ell_i$ meets two 
neighboring twist regions $T,T'$ but does not meet their neighboring twist 
regions, and there is at most one arc in $S\cap P^+$ connecting $T$ and $T'$, 
then there are at least two I-arcs through $T$ or $T'$.
\end{enumerate}
\end{lemma}

\begin{proof}
First note that in both cases there are no O-curves, as such curves must meet all twist 
regions, and $m\ge 4$ implies there are at least three twist  regions in each row.
    
\ref{one twist region I-arcs}: Consider only the arcs of $S\cap P^\pm \cap \ell_i$ 
that belong to a single component of $S\cap T$. If $T$ has $k$ crossings in it, 
then there are $k+1$ arcs meeting $T$ in $P^+$, and $k+1$ arcs meeting $T$ in $P^-$. 
Let $\alpha^+_0,\dots,\alpha^+_k$ (resp.\ $\alpha^-_0,\dots,\alpha^-_k$) be the 
arcs in $P^+$ (resp.\ $P^-$) meeting $T$, ordered from bottom to top on $T$. Without 
loss of generality, we may assume that $T$ is a negative braid. This means 
that the arcs $\alpha_i^+$, after entering $T$ from the left, travel downwards 
in $T$ before leaving from the right side of $T$, as in \Cref{fig:one twist region Iarcs}.

Let $a_0,\dots,a_{k+1}$ and $b_0,\dots,b_{k+1}$ be the endpoints of the arcs 
$\alpha_i^\pm$, ordered as in \Cref{fig:one twist region Iarcs}. 
That is, label so that $\alpha_i^+$ has endpoints $a_{i+1}$ on the left and $b_i$ on the 
right, $i=0, \dots, k$. Set $a_0=b_0$ and $a_{k+1}=b_{k+1}$. It follows that $\alpha_i^-$ 
has endpoints $a_i$ and $b_{i+1}$. 
There exists 
some $1\le i_0\le k+1$ such that $a_j$ is on the bottom boundary of $\ell_i$ if 
and only if $j< i_0$. Similarly, there exists $1\le j_0\le k+1$, such that $b_j$ 
is on the bottom boundary of $\ell_i$ if and only if $j<j_0$.

\begin{figure}
\centering
\includegraphics{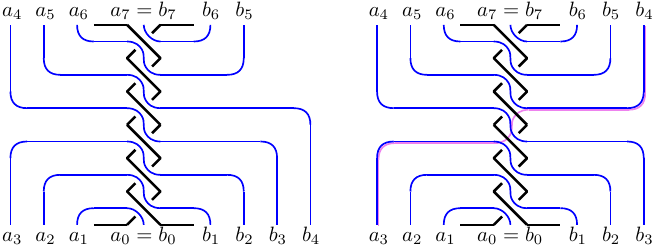}
\caption{A twist region with 6 crossings and the arcs that pass through it. 
On the left $i_0 =4 < j_0 = 5$, and on the right $i_0 =j_0 =4$.}
\label{fig:one twist region Iarcs}
\end{figure}
    
If $i_0<j_0$ then $\alpha^+_{i_0},\alpha^+_{i_0+1}$ are I-arcs, for example as in 
\Cref{fig:one twist region Iarcs} left.  
Similarly, if $i_0>j_0$ then $\alpha^-_{j_0},\alpha^-_{j_0+1}$ are I-arcs. 
If $i_0=j_0$, then $\alpha^+_{i_0},\alpha^-_{i_0}$ are I-arcs, for example as in 
\Cref{fig:one twist region Iarcs}, right. 

\ref{two twist regions I-arcs}:  
Let $\alpha$ be the arc connecting the two twist boxes $T,T'$. Let $\beta$ be the 
arc meeting $T$ opposite to $\alpha$ as in \Cref{fig:two_twist_case}. Up to 
reflections, we may assume that the twist region is a negative braid, and that the arc 
$\beta$ is in $P^+$ and connects to the top component of $\partial \ell_i$ as in the figure. 
It follows that $\beta$ belongs to an I-arc that does not meet $T'$ (see the figure).
A similar argument produces an I-arc meeting only $T'$.
\begin{figure}
    \centering
    \includegraphics{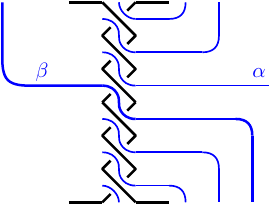}
    \caption{The arcs $\alpha,\beta$.}
    \label{fig:two_twist_case}
\end{figure}

\end{proof}

The following proposition is key to the proof of \Cref{thm: notwistregions}. 
Its proof will be spread over Lemmas~\ref{lem:Ocurve neg chi} through 
\ref{lem:arc one interface neg chi}. 

\begin{proposition}\label{inequality for inbetween}
Assuming $m\geq 4$ and 4-highly twisted, for all $1\le i \le n$, we have $\chi_i\le -1$. 
Moreover, if $S$ passes through 
a twist region in $\ell_i$ then the inequality is strict.
\end{proposition}

The proof relies on Lammas \ref{lem:Ocurve neg chi} to \ref{lem:arc one interface neg chi} 
below, where we assume $1\leq i\leq n$. We remark that for a vertical sphere 
$\chi_i = -1$ for all $1\le i \le n$.

\begin{lemma}\label{lem:Ocurve neg chi}
Assume $m\geq 4$ and the diagram is 4-highly twisted. 
Suppose $S\cap P^{\pm} \cap \ell_i$ contains an O-curve. Then $\chi_i<-1$ (and $i$ is even). 
\end{lemma}

Note \Cref{lem:Ocurve neg chi} is one place where we use the hypotheses $m\geq 4$ and 
$4$-highly twisted.

\begin{proof}
First, if there exists an O-curve, then $i$ is even by  
\ref{no scc in layers}.

Let us consider each side $P^+,P^-$ separately. Assume that there are 
$o^+\ge 1$ O-curves in $P^+$. 

\begin{equation}\label{contribution of O-curves}
\sum_{\gamma\text{ is an O-curve in }P^+}\chi_+(\gamma) = o^+
\end{equation}

Let $T_1,\dots,T_m$ be the twist boxes in $\ell_i$ ordered from left to right. 
Each twist box $T_j$ has $k_j$ bubbles, and the surface $S$ meets it with $s_j$ 
components. That is, there are $k_j\cdot s_j$ saddles in $T_j$, and out of 
them $(k_j-2)\cdot s_j$ are internal saddles. Therefore,

\begin{equation}\label{number of internal bubbles} 
\#\{\text{internal saddles in $\ell_i$}\} = \sum_{j=1}^m (k_j-2)\cdot s_j
\end{equation}

Now, in $P^+$, there are $k_j\cdot s_j$ arcs or curves of $S\cap P^+ \cap \ell_i$ 
entering $T_j$ through its left side, and a similar number leaving it from the right. 
There are $s_j$ curves entering through the bottom, and a similar number leaving 
through the top bubble. Each one of the O-curves passes through $T_j$, enters through 
the left side of $T_j$ and exits through the right side. There are at most 
$(k_j -1)s_j$ such components. Thus,
\begin{equation}\label{O-curves bound}
    o^+ \le \min\{(k_j -1)s_j \;:\;1\le j \le m\}
\end{equation}
    
Any U-arc whose left endpoint is free must enter $T_1$ from the left. Since there 
are $(k_1-1)s_1$ components entering $T_1$ from the left, and $o^+$ of them are O-curves, 
there are at most $((k_1-1)s_1 - o^+)$ U-arcs whose left endpoint is free. Similarly, 
at most $((k_m-1)s_m - o^+)$ U-arcs have free right endpoint. When 
$\alpha$ is not an O-curve, we recall that 
\[ \chi_+(\alpha) \leq \tfrac14\#\{\text{ free endpoints of }\alpha\}, \]
with equality only if $\alpha$ is extremal.
And so
\begin{equation}\label{free endpoints}
\sum_{\alpha\in \{\text{U-arcs of }S\cap P^+\cap \ell_i\}} \chi_+(\alpha) \le 
\tfrac14((k_1-1)s_1 - o^+) + \tfrac14((k_m-1)s_m - o^+)
\end{equation}

Combining the above, by \eqref{contribution of O-curves}, \eqref{number of internal bubbles}, 
and \eqref{free endpoints}
\begin{align*}
\chi_i^+ & := \sum_{\alpha \subset S\cap P^+\cap \ell_i}\chi_+(\alpha) -\tfrac12\#
\{\text{internal saddles in $\ell_i$}\} \\
&\le o^+ + \tfrac14((k_1-1)s_1 - o^+) + \tfrac14((k_m-1)s_m - o^+) - \tfrac12
\sum_{j=1}^m (k_j-2)\cdot s_j
\end{align*}
Simplifying, 
\[ \chi_i^+ \leq 
\tfrac12 o^+ + \tfrac14(3-k_1)s_1 + \tfrac14(3-k_m)s_m - \tfrac12 \sum_{j=2}^{m-1}(k_j-2)s_j
\]
Since there are O-curves $s_j\ge 1$ for all $1\le j \le m$. Also by assumption $k_1, k_m\geq 4$ 

\begin{align*}
\chi_i^+ &< \tfrac12 o^+ - \tfrac12 \sum_{j=2}^{m-1}(k_j-2)s_j  \\
&  \leq \tfrac12 o^+ - \tfrac12(m-2) \min_{\{2\le j\le m-1\}}\{(k_j-2)s_j\} 
\end{align*}

By \eqref{O-curves bound}, 
\begin{align*}
\chi_i^+ &< \tfrac12 \min_{\{1\le j \le m\}}\{(k_j-1)s_j\} - \tfrac12(m-2) 
\min_{\{2\le j\le m-1\}}\{(k_j-2)s_j\} \\
& \leq \tfrac12\max_{\{2\le j \le m-1 \}} \{(k_j-1)s_j-(k_j-2)(m-2)s_j\} \\
& \leq \tfrac12\max_{\{2\le j \le m-1\}} \{ -s_j(k_j-2)(m-3)+s_j\}\\
&\le \tfrac12\max_{\{2\le j \le m-1\}} \{ s_j(1-(k_j-2)(m-3))\}
\end{align*}

Assuming $m\geq 4$ and $k_j\geq 4$, and using $s_j\ge 1$ again for all $1\le j \le m$, we get
\[ \chi_i^+ < -\tfrac12 \]

Now consider $P^-$. 
There may not be an O-curve in $P^-\cap \ell_i$, but \cref{contribution of O-curves} still 
holds, with $o^+$ replaced by $o^-$, the number of O-curves in $P^-$.  
Similarly equations~\eqref{number of internal bubbles} and~\eqref{O-curves bound} 
still hold with $o^+$ replaced by $o^-$. 
Then the same computation as above shows
\[
\chi^-_i: = \sum_{\alpha\subset S\cap P^-\cap \ell_i}\chi_+(\alpha) -\tfrac12\#\{\text{internal 
saddles in $\ell_i$}\}< -\tfrac12
\]
Summing them together we get $\chi_i = \chi^+_i+\chi^-_i<-1$, by 
\Cref{euler of layer is sum over arcs}.
\end{proof}

\begin{lemma}\label{lem:arc two interface neg chi}
Suppose $S\cap P^{\pm}\cap \ell_i$ does not contain an O-curve, and suppose every 
arc of $S\cap P^{\pm}\cap \ell_i$ either has two interface endpoints, or has two free 
endpoints and does not meet a twist region. Then $\chi_i\leq -1$, and $\chi_i<-1$ if 
$S$ passes through a twist region. 
\end{lemma}

\begin{proof}
By \ref{no trivial U-arcs}, every U-arc must have two interface endpoints. 
It follows from \Cref{Lem:ChiPlusFormula} that $\chi_+(\alpha)\leq 0$ for every 
arc $\alpha$. 

\textbf{Case 1.} $S\cap \ell_i$ does not pass through twist boxes. 
First we find an I-arc in $S\cap P^\pm \cap \ell_i$. 
If there is no such I-arc, then there is a horizontal line $H$ running across 
$\ell_i$ that does not intersect $S\cap P^{\pm}$. However, the diagrams above and 
below $H$ are both trivial tangles on $m$ strands. There is no super-incompressible 
sphere contained in the complement of such a tangle in a ball. So there must be an I-arc. 
    
This I-arc belongs to a closed curve, so it must intersect $\ell_i$ in another I-arc. 
By \ref{no double intersection outside}, at least one of them has two interface endpoints. 
Without loss of generality, suppose this I-arc lies in $S\cap P^+\cap \ell_i$; 
call it $\alpha^+$. 

Since $S$ does not pass through twist regions, $S\cap P^-\cap \ell_i$ must contain an 
identical I-arc $\alpha^-$. So there are two I-arcs $\alpha^+,\alpha^-$ in $\ell_i$ on 
$P^+,P^-$ respectively, each with two interface endpoints. Thus, 
$\chi_+(\alpha^+)=\chi_+(\alpha^-)=-\tfrac12$.
Altogether, by \Cref{euler of layer is sum over arcs}, 
\[ \chi_i = \sum_\alpha \chi_+(\alpha) - \#\{\text{internal saddles}\} \le 
    \chi_+(\alpha^+)+\chi_+(\alpha^-)= -1. \]

\textbf{Case 2.} $S\cap \ell_i$ passes through a twist box $T$. 
By the hypothesis that the plat is 4-highly twisted, it must contain at least two 
internal saddles. Since each $\chi_+(\alpha)\leq 0$, it follows that
\[ \chi_i =\sum_\alpha \chi_+(\alpha) - \#\{\text{internal saddles}\} \le 0 - 2 <-1 \qedhere \]
\end{proof}

\begin{lemma}\label{lem:arc one interface twregion neg chi}
Assume there are no O-curves and there exists an arc $\alpha$ with at most one 
interface endpoint that passes through a twist region. Then $\chi_i<-1.$
\end{lemma}

\begin{proof}
The arc $\alpha$ must have at least one free endpoint, so $i$ is even. 
Enumerate the twist boxes $T_1,\dots,T_m$, and denote by $k_1,\dots,k_m$, $s_1,\dots,s_m$ 
the number of crossings in $T_1,\dots,T_m$ and the number of components of 
$S\cap T_1,\dots,S\cap T_m$ respectively.

Since $\alpha$ has at least one free endpoint, it must (a) pass through $T_1$ and have a 
free endpoint to its left, or (b) pass through $T_m$ and have a free endpoint to its 
right (or both). Let $F_L$ be the number of free endpoints corresponding to (a), and 
$F_R$ be the number of free endpoints corresponding to (b). Note that each free endpoint 
is actually counted twice, once as an endpoint of a curve in $P^+$ and once in $P^-$. 
Then $F_L \le 2k_1 \cdot s_1$ and $F_R \le 2k_m\cdot s_m$. Let us denote by 
$G_L = 2k_1s_1 - F_L$ (resp.\ $G_R = 2k_m s_m - F_R$) the arcs meeting the left 
side of $T_1$ (resp.\ the right side of $T_m$) and not ending in a free endpoint. 
    
Note that 
\begin{multline}
\tfrac14 F_L  - \#\{\text{internal saddles of }T_1\} = \tfrac14 (2k_1 \cdot s_1) - 
\tfrac14 G_L  - (k_1-2)s_1 
\\
= (-\tfrac12 k_1 +2)s_1 - \tfrac14 G_L \le - \tfrac14 G_L     
\end{multline}
where the last inequality follows from the assumption $k_1\ge 4$.
Similarly, the assumption $k_m\ge 4$ implies
\[
\tfrac14 F_R  - \#\{\text{internal saddles of }T_m\} \le - \tfrac14 G_R
\]

We need to compute $\chi_i = \sum \chi_+(\alpha) - \#\{\text{internal saddles in }\ell_i\}$. 
By \Cref{euler of layer is sum over arcs}, \Cref{Lem:ChiPlusFormula} and the fact that 
there are no O-curves, this is given by
\begin{align*}
\chi_i &= \sum \chi_+(\alpha) - \#\{\text{internal saddles in }\ell_i\} \\
&= \tfrac14 \#\{\text{free endpoints of arcs}\} - \tfrac12 \#\{\alpha\text{ non-extremal}\} - \#
\{\text{internal saddles in }\ell_i\} \\
&= \tfrac14 (F_L + F_R) -\tfrac12\#\{\alpha\text{ non-extremal}\} - \#\{\text{internal 
saddles of }T_1,T_m\} - \\
& \hspace{3cm}  \#\{\text{internal saddles of }T_2,\dots, T_{m-1}\}
\end{align*}
The above inequalities relating $F_L, F_R$ and $G_L, G_R$ imply
\begin{equation}\label{Eqn:ChiBound}
\chi_i \leq -\tfrac14 (G_L + G_R) - \tfrac12\#\{\alpha\text{ non-extremal}\} 
- \#\{\text{internal saddles of }T_2,\dots, T_{m-1}\}
\end{equation}

\textbf{Case 1.} Suppose some $s_j \in \{s_2, \dots, s_{m-1}\}$ is nonzero. Then by the 
assumption that $k_j\geq 4$, there are at least two internal saddles. In this case the 
third term in the inequality \eqref{Eqn:ChiBound} is then at most $-2$, with the other 
two terms non-positive, so we have $\chi_i \leq -2 < -1$. 

\textbf{Case 2.} If all of $s_2= \dots = s_{m-1} =0$, \Cref{Eqn:ChiBound} becomes
\begin{equation}\label{Eqn:ChiBound2}
\chi_i \leq -\tfrac14 (G_L+G_R) - \tfrac12\#\{\alpha\text{ non-extremal}\}
\end{equation}
We divide into further cases.

Consider an arc in $S\cap P^+\cap \ell_i$ that runs from the left side of $T_1$, around 
the annulus $\ell_i$ avoiding all other twist regions, to connect to the right side of $T_m$. 
The existence of any such arc implies there is another such arc on $P^-$. These are each 
counted in both $G_L$ and $G_R$. Thus any such arc contributes $4$ to $G_L+G_R$.  
    
\textbf{Case 2.A.} There are at least two such arcs in $S\cap P^+\cap \ell_i$. 
Then the first term of \Cref{Eqn:ChiBound2} is at most $-\tfrac14\cdot 8 \leq -2 < -1$, 
and the second is nonpositive, so we have the result in this case. 

\textbf{Case 2.B.} There is exactly one such arc in $S\cap P^+\cap \ell_i$. Then 
\Cref{finding I arcs}~\ref{two twist regions I-arcs} implies that there are (at least) 
two I-arcs. This means that the first term of \Cref{Eqn:ChiBound2} is at most 
$-\tfrac14\cdot 4$ and the second is at most $-\tfrac12 \cdot 2$, so the sum is at most
\[ \chi_i \leq -1 - 1 < -1 \]

\textbf{Case 2.C.}
Suppose there are no such arcs. The case $s_1=s_m=0$ does not occur because of our 
assumptions, so suppose without loss of generality that $S\cap P^\pm \cap \ell_i$ 
meets $T_1$. Then, by \Cref{finding I arcs} \ref{one twist region I-arcs}, there 
are at least two I-arcs $\alpha_1$, $\alpha_2$ passing through $T_1$.

Each of $\alpha_i$ belongs to a simple closed curve of $S\cap P^\pm$ that intersects 
$\ell_i$ in another I-arc, say $\beta_i$. The curves $\alpha_1$, $\alpha_2$, $\beta_1$,
$\beta_2$ must all be distinct by \Cref{menasco assumptions}. This means we have four 
distinct I-arcs in this case. 
Then the second term of \Cref{Eqn:ChiBound2} is at most $-\tfrac12 \cdot 4 = -2$, and 
the first term is non-positive, so $\chi_i < -1$. 
\end{proof}

\begin{lemma}\label{lem:arc one interface neg chi}
Suppose $S\cap P^{\pm} \cap \ell_i$ contains no O-curve, and every arc with at most one 
interface point does not pass through a twist region. Suppose there is an arc $\alpha$ 
with one free endpoint, such that $\alpha$ does not pass through a twist region. 
Then $\chi_i <-1$. 
\end{lemma}

\begin{proof}
By \ref{no trivial U-arcs}, $\alpha$ is an I-arc. It has exactly one free endpoint and one 
interface endpoint. 

Since it does not pass through a twist region, $\alpha$ appears both in $P^+$ and $P^-$. 
Let us call these instances $\alpha_+,\alpha_-$ respectively. The arc $\alpha_+$ is 
a subarc of a simple closed curve in $P^+$ that must intersect $\ell_i$ again 
in some other I-arc $\beta_+$. By \ref{no double intersection outside}, $\beta_+$ 
must have two interface endpoints. Similarly one defines $\beta_-$, and shows 
that it has two interface endpoints. 

There are no O-curves. There also cannot be U-arcs with free endpoints, since 
by \ref{no trivial U-arcs} these would meet a twist region, and no such arcs exist by 
hypothesis (alternatively, they are handled by \Cref{lem:arc one interface twregion neg chi}). 
Thus by \Cref{Lem:ChiPlusFormula} there are no arcs or curves with positive $\chi_+$, 
and so by \Cref{euler of layer is sum over arcs}
\[ \chi_i\le \sum_{\gamma}\chi_+(\gamma) \le\chi_+(\alpha_+) + \chi_+(\beta_+) + 
    \chi_+(\alpha_-)+\chi_+(\beta_-) \le -\tfrac14 -\tfrac12 -\tfrac14 -\tfrac12 <-1 \qedhere \]
\end{proof}

\begin{proof}[Proof of \Cref{inequality for inbetween}]
If $i$ is odd, all arcs $S\cap P^{\pm} \cap \ell_i$ are as in \Cref{lem:arc two 
interface neg chi} and the result holds by that lemma. 

If $i$ is even, then either there is an O-curve and the result holds by 
\Cref{lem:Ocurve neg chi}, or all arcs are as in \Cref{lem:arc two interface neg chi} 
and the result holds by that lemma, or there is an arc with at most one interface 
endpoint. In that case, the result holds by Lemma~\ref{lem:arc one interface neg chi}
or~\ref{lem:arc one interface twregion neg chi}. 
\end{proof}

\begin{lemma}\label{inequality for top and bottom}
For the top two levels, $\chi_0+\chi_1\le -\tfrac12$. Similarly, for the 
bottom levels, $\chi_n+\chi_{n+1}\le -\tfrac12$. 
If $S$ passes through a twist region in 
$\ell_1$ (or resp.\ $\ell_n$) then the inequality is strict.
\end{lemma}

Again, we remark that for a vertical sphere there is equality.

\begin{proof}
We prove the result for $\ell_0$ and $\ell_1$; the case of $\ell_n$ and 
$\ell_{n+1}$ is similar. By \Cref{lem:arc two interface neg chi},  
$\chi_1\le -1$, and this inequality is strict if $S$ passes 
through a twist region in $\ell_1$. Thus, if $\ell_0$ contains no U-arcs 
with exactly one interface point, then we are done, since in this 
case $\chi_0\le 0$.
    
Otherwise, let $\alpha_1,\dots,\alpha_s$, $s\ge 1$, be all the U-arcs 
with one interface endpoint in $\ell_0$. Note that in fact $s\ge 2$, 
because each arc of this form appears both in $P^+$ and $P^-$.
    
The interface endpoint of $\alpha_i$ is the endpoint of an arc 
$\beta_i$ in $\ell_1$ on the same side $P^+$ or $P^-$ as $\alpha_i$. 
No two $\alpha_i$ correspond to the same arc $\beta_i$ since otherwise 
they would belong to the same curve, in contradiction to 
\ref{no double intersection outside}.
    
The arc $\beta_i$ has two interface endpoints, and is not an extremal 
U-arc because the other endpoint of $\alpha_i$ (which is free) meets an I-arc 
in $\ell_1$, so the curve containing $\beta_i$ meets $\ell_2$, and so $\beta_i$ 
does not contain the global minimum of this curve. Hence $\chi_+(\beta_i)= -\tfrac12$.
Together $\chi_+(\alpha_i) + \chi_+(\beta_i)= -\tfrac14$. 
Since all arcs in $\ell_0,\ell_1$ beside $\alpha_1,\dots,\alpha_n$ 
have $\chi_+\le 0$, \Cref{euler of layer is sum over arcs} implies
$$\chi_0+\chi_1 \le \sum_{i=1}^s(\chi_+(\alpha_i)+\chi_+(\beta_i)) 
- \#\{\text{internal saddles}\} \le -\tfrac14 s \le -\tfrac12$$

Moreover, if $S$ passes through a twist region then there are internal
saddles. This renders the inequality strict. 
\end{proof}

\begin{proof}[Proof of \Cref{thm: notwistregions}]
Let $S$ be a super-incompressible sphere with at most $p=n+1$ punctures satisfying 
\ref{no vertical intersection}--\ref{no double intersection outside}.
Then $2-p \leq \chi(S)$. 
By \Cref{euler by layer}, using the fact that $n\geq 2$, 
\begin{align*}
2-p&\leq \chi(S) \\
&= \sum_{i=0}^{p}\chi_i\\
    &= \left(\chi_0 + \chi_1\right)h + \sum_{i=2}^{n-1} \chi_i + \left(\chi_n + 
    \chi_{n+1}\right) \\
    &\le -\tfrac12 +\cdot (n-2)\cdot (-1) - \tfrac12 
    \\&= 2-p,
\end{align*}
where the last inequality is by \Cref{inequality for inbetween,inequality for top and bottom}. 
This implies that $\chi(S)=2-p$, and thus $S$ has $n+1$ punctures. The sphere $S$ cannot 
pass though a twist region, as otherwise, by \Cref{inequality for inbetween,inequality 
for top and bottom}, 
the inequality above would be strict and result in a contradiction.
\end{proof}

\section{Collections of vertical spheres}\label{sec:Collections}
In this section, we use the above results for  collections of super-incompressible 
vertical spheres, to deduce the image of a vertical sphere under the isotopy $\varphi$. 
We consider a ``maximal'' collection (see Definition \ref{def:maximal}) of such 
vertical spheres which satisfy the additional property that any two 
successive spheres bound a single twist region between them. We then show that 
such a maximal collection is mapped, by $\varphi$, to a similar maximal 
collection and that the twist regions in between successive spheres, in such a 
collection, are uniquely determined as well. This is enough to show that all  
twist regions that are not the first or last in their row are independent 
of the  plat projection. 

\vskip7pt

Throughout this section we will work under the assumption of \Cref{thm:UniquePlat}. 
More precisely:

\begin{assumption}\label{ass: basic assumption}
    Let $K,K'$ be two $4$-highly twisted plat projections, 
    on the plane $P$, of widths $m,m'\ge 4$ and heights $n,n'$ respectively. 
    Assume that $K$ and $K'$ represent the same knot $\mathcal{K}$, and 
    let $\varphi\from (S^3,K')\to (S^3,K)$ be an ambient isotopy taking $K'$ to $K$. 
\end{assumption}

\begin{theorem}\label{cor: verticaltovertical}  
Under \Cref{ass: basic assumption}, if $S'\subset(S^3, K')$ is a vertical 
sphere  then $\varphi (S')$ is a vertical sphere in $(S^3, K)$. In particular, 
 the diagrams $K$ and $K'$ have the same height, $n=n'$.
\end{theorem}

\begin{proof} 
Without loss of generality $n'\le n$  (otherwise consider $\varphi^{-1}$).
As $S'$ is a vertical sphere, it is super-incompressible by \Cref{Thm:super incompressible}. 
It has $n'+1$ intersections with $K'$. The ambient isotopy 
$\varphi$ takes $K'$ to $K$. Hence the image $S \ssm \NN(K) =  \varphi(S' \ssm \NN(K))$ 
is super-incompressible in $S^3 \ssm \NN(K)$ with $n'+1 \leq n+1$ intersections with $K$. 
Theorem~\ref{thm: notwistregions} tells us that $S$ does not pass through twist 
regions. Hence by Theorem~\ref{thm:Verticalspheres} the sphere $S$ is a vertical  
$2$-sphere. In particular, $n=n'$.
\end{proof}

Consider now isotopies of collections of vertical $2$-spheres 
within a diagram. 

\begin{definition}\label{def:maximal}
A disjoint collection $\calS$ of vertical spheres is \emph{maximal} if any 
two vertical spheres in $\calS$ are non-isotopic and the set $\calS$ is maximal
with respect to this property. 
\end{definition}

\begin{remark}
A maximal collection of vertical 2-spheres $\calS$ has the following properties. 
\begin{enumerate}
\item The collection is finite. Its spheres can be enumerated $S_1,\dots,S_r$ 
from left to right such that any consecutive pair $S_j,S_{j+1}$ is not separated 
by any other sphere $S\subset\calS$.
\item If $S_j,S_{j+1}$ are consecutive then the region of the diagram in 
between them contains exactly one twist region.
\item The first sphere is the leftmost sphere $S_1 = S(1,\dots,1)$.
The last sphere is the rightmost sphere $S_r=S(m-2,m-1,m-2,\dots,m-2)$ 

\item The number $r$ of spheres in a maximal collection is one more 
than the number of twist regions that are not leftmost or rightmost 
in their level. So it can be expressed in terms of $m$ and $n$, namely, 
\begin{equation}
\label{size of maximal collection}
r = \lceil\tfrac{n}{2}\rceil(m-3) + \lfloor\tfrac{n}{2}\rfloor(m-2) + 1
\end{equation} 
\end{enumerate}
\end{remark}

\begin{lemma}\label{lem: maximal goes to maximal}
Under \Cref{ass: basic assumption}, 
suppose $\calS'$ is a maximal disjoint collection of vertical spheres 
in $K'$. Then the image $\varphi(\calS')$ is a maximal disjoint collection 
of vertical spheres in $K$. In particular, $K$ and $K'$ have the same width.
\end{lemma}

\begin{proof}
By \Cref{removing unwanted intersections} and \Cref{thm:Verticalspheres},
we may isotope $\varphi$ so that the image $\calS = \varphi(\calS')$ is a
collection of vertical 2-spheres. Note that any two vertical 2-spheres in 
$\calS$ are non-isotopic, so the number of spheres in $\calS$ 
is at most that of a maximal collection for $K'$. Similarly, applying 
$\varphi^{-1}$ to a maximal collection of vertical $2$-spheres in $K$ 
implies $r\leq r'$. Therefore, $r=r'$, and the collection
$\calS=\varphi(\calS')$ is maximal in $K$. Since $n=n'$ by 
\Cref{cor: verticaltovertical}, it follows that $m=m'$. 
\end{proof}

Up to flipping $K$ along a vertical axis of $P$, and applying a (layer-preserving) 
isotopy to the diagrams $K,K'$, we may assume:
\begin{assumption}\label{assum:FlipVertical}
In addition to \Cref{ass: basic assumption}, assume that:
\begin{enumerate}
\item $K,K'$ are drawn on the same layers $\ell_0,\dots,\ell_n$.
\item The collection $S_1,\dots,S_n$ is a maximal collection of vertical spheres for 
both $K$ and $K'$; they are ordered from left to right, and intersect $P$ in vertical 
lines (except outside some large ball containing both $K$ and $K'$). 
\item For all $1\le j\le r$, we have $ \varphi(S_j) = S_j$ and $K\cap S_j = K'\cap S_j$.
\end{enumerate}
\end{assumption}

Note that while we do assume that $\varphi(S_j) = S_j$, we do not yet know that $\varphi$ 
is the identity on $S_j$. More so, we have $\varphi(K\cap S_j)=K'\cap S_j = K\cap S_j$ 
and so $\varphi$ induces a permutation on $K\cap S_j$ but we do not yet know that this
permutation is the identity.

\begin{definition}\label{def:VerticalSlab} \hfill 
\begin{enumerate}

\item For $1\le j \le  r$, let the points of intersection $K\cap S_j=K'\cap S_j$ 
be $x_{0,j},\dots,x_{n,j}$ ordered from top to bottom. See \Cref{fig:vertical slab} 
for reference. 

\item For $1\le j \le r$ and $1\le i\le  n$ let $\alpha_{i,j}$ denote the 
\emph{vertical arc} connecting $x_{i-1,j}$ and $x_{i,j}$ on $S_j\cap P$.

\item For $1\le j <r$, let $W_j$ denote the \emph{vertical slab} region 
between $S_j$ and $S_{j+1}$, homeomorphic to $S^2\times [-1,1]$. 

\item For $1\le j <r$, the vertical slab $W_j$ contains a unique twist region 
of $K$ (resp.\ of $K'$). Let $k_j$ (resp.\ $k_j'$) denote the number of crossings 
in this twist region.
Let $p_j$ (resp.\ $p_j'$) be the index of the layer $\ell_{p_j}$ (resp.\ $\ell_{p'_j}$) 
in which the twist region of $K\cap W_j$ (resp.\ $K'\cap W_j$) lies.

\item For $1\le j <r$, the intersection $K\cap W_j$ consists of $n+1$ arcs. Two of 
these arcs are \emph{twisted} and connect 
the points $x_{p_j-1,j},x_{p_j-1,j+1}$ to the points $x_{p_j,j},x_{p_j,j+1}$. 
All other arcs of $K\cap W_j$ are \emph{horizontal arcs} connecting $x_{i,j}$ and 
$x_{i,j+1}$ across $W_j$ for all $i\ne p_j,p_j-1$. Similarly we define 
the twisted and horizontal arcs for $K'\cap W_j$.
\end{enumerate}
\end{definition}

\begin{figure}
\centering
\bigskip
\begin{overpic}[height = 5cm]{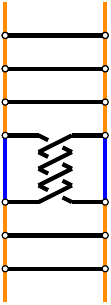}
        \put(13,103){$W_j$}
        
        \put(-2,103){\color{orange}$S_j$}
        \put(30,103){\color{orange}$S_{j+1}$}
        
        \put(-12,88)
        {$x_{0,j}$}
        \put(37,88)
        {$x_{0,j+1}$}
        
        \put(-12,78)
        {$x_{1,j}$}
        \put(37,78)
        {$x_{1,j+1}$}
        
        \put(-12,68)
        {$x_{2,j}$}
        \put(37,68)
        {$x_{2,j+1}$}
        \put(-12,57)
        {$x_{3,j}$}
        \put(37,57)
        {$x_{3,j+1}$}

        \put(-16,44){\color{blue} $\alpha_{4,j}$}
        \put(41,44){\color{blue} $\alpha_{4,j+1}$}
        
        \put(-12,33)
        {$x_{4,j}$}
        \put(37,33)
        {$x_{4,j+1}$}
        
        \put(-12,22)
        {$x_{5,j}$}
        \put(37,22)
        {$x_{5,j+1}$}
        \put(-12,10)
        {$x_{6,j}$}
        \put(37,10)
        {$x_{6,j+1}$}
\end{overpic}
\caption{The vertical slab $W_j$ between $S_j$ and $S_{j+1}$. The unique twist 
box of $K$ in $W_j$ has 4 crossings --- so $k_j = 4$ --- and is located at the 
layer $\ell_4$ --- so $p_j = 4$. The vertical arcs $\alpha_{4,j}$ and $\alpha_{4,j+1}$ 
are shown in blue.} 
\label{fig:vertical slab}
\end{figure}
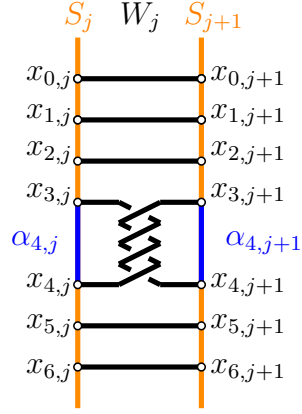

\begin{lemma}\label{lem: twists between consecutive} For $K$, $K'$ satisfying the conditions of 
\Cref{assum:FlipVertical} we have:
\begin{enumerate}

\item For all $1\le j <r$, $\varphi$ maps the twisted pair of arcs in $K'\cap W_j$ 
to the twisted pair in $K\cap W_j$.

\item  For all $1\le j <r$, the number of crossings $k_j = k_j'$.

\item Up to an isotopy, for all $1\le j \le r$, 

\begin{enumerate}[label = (3\alph*)]
    \item  $\varphi(\alpha_{p'_j,j})=\alpha_{p_j,j}$ (possibly with reverse 
    orientation), and 
    
    \item the map $\varphi$ maps the horizontal segments connecting the endpoints of the 
    twisted pair in $K'\cap W$ to the horizontal segments connecting the endpoints 
    of the twisted pair in $K\cap W$. That is, it maps the horizontal segments 
    $[x_{p'_j-1,j}, x_{p'_j-1,j+1}]$ and $[x_{p'_j,j},x_{p'_j,j+1}]$ to the
    horizontal segments $[x_{p_j-1,j}, x_{p_j-1,j+1}]$ and $[x_{p_j,j},x_{p_j,j+1}]$ 
    (possibly in different order).
\end{enumerate}
\end{enumerate}
\end{lemma}

\begin{proof}
Let $1\le j <r$.  
Let us suppress the index $j$ by denoting
\begin{itemize} 
\item $W := W_j$;
\item $S^- := S_j$ and $S^+ := S_{j+1}$;
\item $k= k_j$, $k'=k_j'$, $p=p_j$, $p'=p'_j$;
\item $\alpha^- =  \alpha_{p,j}$ and $\alpha^+=\alpha_{p,j+1}$;
\item $\alpha'^- = \alpha_{p',j}$ and $\alpha'^+=\alpha_{p',j+1}$;
% \item $x^-_i = x_{i,j}$ and $x^+_j = x_{i,j+1}$.
\end{itemize}
Finally, let $a,b$ (resp.\ $a',b'$) be the twisted pair of arcs in $K\cap W$ 
(resp.\ $K'\cap W$).

\vskip7pt
(1) The arcs $a',b'$ are connected to each other by $\alpha^-$
on $S^-$ and $\alpha^+$ on $S'^+$. The union of the arcs $a'\cup b' \cup
\alpha'^-\cup \alpha'^+$ is a $(2,k')$ torus knot or link. 

Moreover, the pair of strands $a',b'$ is the only pair of arcs in $K'\cap W$ 
for which this happens: Closing any other pair of arcs of $K'\cap W$ by an arc 
in $S^-$ and an arc in $S^+$ to form a closed knot or link, will result in 
the unknot. This topological property is preserved by an ambient isotopy 
$\varphi$. Hence, $\varphi$ will map torus link to a torus link and unknots 
to unknots. This proves that $\varphi$ maps $a',b'$ to $a,b$.
\vskip7pt
(2) The torus knot $(2,k')$ discussed in (1) is mapped by $\varphi$ to the torus 
knot $(2,k)$. It follows that $k=k'$. 
\vskip7pt
(3)(a) Let $D'^-, D'^+$ be small disk 
neighborhoods of $\alpha'^-,\alpha'^+$ in $S'^-,S'^+$, respectively.
Let $A'$ be the annulus in the interior of $W \ssm K'$ whose boundary 
is $\partial A' =  \partial D'^-\cup \partial D'^+$. 
Similarly, define $D^{\pm}$ as disk neighborhoods of $\alpha^{\pm}$, 
and an annulus $A$ in the interior of $W\ssm K$ with boundary  
$\bdy A = \bdy D^-\cup \bdy D^+$.

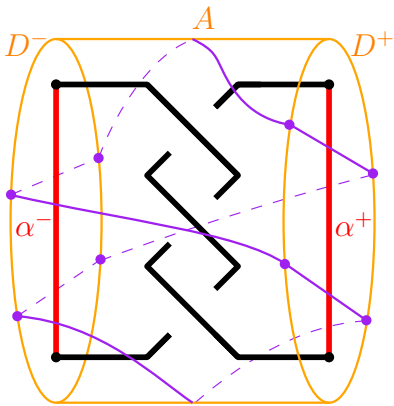
\begin{figure}
\centering
\bigskip
\begin{overpic}[width=5cm]{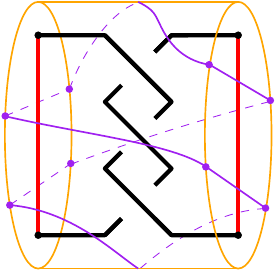}
        \put(3,45){\color{red} $\alpha^-$}
        \put(88,45){\color{red} $\alpha^+$}
        \put(50,100){\color{orange} $A$}
        \put(0,93){\color{orange} $D^-$}
        \put(92,93){\color{orange} $D^+$}
\end{overpic}
\caption{The tangle $(A\cup D^- \cup D^+, a\cup b)$, the vertical arcs
$\alpha^-,\alpha^+$ and the boundary of the only non-trivial compressing 
disk shown in purple.} 
\label{fig:compressing disk for tangle}
\end{figure}

Let $B'$ (resp.\ $B$) be the 3-ball bounded by the sphere $A' \cup D'^-\cup D'^+$  
(resp.\ $A \cup D^-\cup D^+$) in $W$. The tangle $(B', a'\cup b')$ 
(resp.\ $(B,a\cup b)$)  has only one non-trivial compressing disk, and 
its boundary is isotopic to the purple curve in
Figure~\ref{fig:compressing disk for tangle}. 
That is, it is a curve of slope $1/k$ on the corresponding $4$-punctured sphere. 
Thus this boundary curve has at least $|k|+1$ intersections with $A'$ (resp.\ $A$). 
In particular, it follows that any disk whose boundary intersects the 
annulus $A'$ in a closed curve, or in at most two arcs, is isotopic to 
a disk on $\partial B'=A' \cup D'^-\cup D'^+$. 
A similar statement holds for $A$. 

Without loss of generality,  we assume that $\varphi(A')$ and $A$ intersect 
transversely. Consider the intersection curves and arcs of $\varphi(A') \cap A$. 

Consider first inessential closed curves of $\varphi(A') \cap A$. An innermost 
such curve in $A$ bounds a compressing disk $E$ for $\varphi(A')$. By the
observation above, this disk is isotopic to a disk on $\varphi(A')$. Hence, 
innermost intersection curves can be removed by an isotopy.
The same argument can be applied to an innermost arc of $\varphi(A') \cap A$ 
that has both endpoints on the same boundary of $A$. It applies as well 
to an innermost pair of arcs whose endpoints are on different boundary 
components of $A$, for in these cases, the disk bounded by arcs of 
$\varphi(A')\cap A$ intersects $A$ in at most two arcs. 
By iterating this argument we conclude that $\varphi$ can be isotoped in $W$ so 
that $\varphi(A')$ and $A$ intersect only along curves which are parallel to 
the core curve of each annulus. At this point we already know that 
$\partial A \cap \varphi(\partial A')=\emptyset $. In particular, the disks 
$D^-$ and $\varphi(D')$ have disjoint boundaries. Note that by (2) the two 
punctures of $D'^-$ are sent by $\varphi$ to the two punctures of $D^-$, 
and so $D^-\cap \varphi(D'^-)\ne \emptyset $. It follows that either 
$D^- \subseteq \varphi(D'^-)$ or $D^- \supseteq \varphi(D'^-)$. 
Each disk contains only one non-trivial arc, namely $\alpha^- \subseteq D^-$ 
and $\alpha'^- \subseteq D'^-$. It follows that up to an isotopy we may 
assume that $\varphi(D'^-)=D^-$ and $\varphi(\alpha'^-) = \alpha^-$. 
Similarly, for $D^+$ and $\varphi(D'^+)$, we get that up to an isotopy 
$\varphi(D'^+) = D^+$ and $\varphi(\alpha'^+) = \alpha^+$. This proves (3a).

We now have $\partial A = \varphi(\partial A')$ and any other intersection 
curve is parallel to these curves on $A$ and $\varphi(A')$.
Suppose $c\subseteq A\cap \varphi(A')$ is such an intersection that is not in $\partial A$.
Let $U$ be an annulus in $A$ such that $U\cap \varphi(A') = \partial U$. Then there 
exists an annulus $U'\subseteq \varphi(A')$ such that $\partial U = \partial U' = U\cap U'$. 
Their union $U\cup U'$ forms a torus in $W_j \ssm K$. Since there are no incompressible 
tori in $W_j\ssm K$ it follows that up to an isotopy we can isotope one annulus 
through the other and remove the intersection $c$. Iterating this we get 
that $A\cap \varphi(A') = \partial A = \varphi(\partial A')$. 
Their union is a torus, and using an isotopy we may assume that 
$A=\varphi(A')$, and so also the ball $B = \varphi(B')$.

Consider the horizontal arcs $[x_{p'-1,j},x_{p'-1,j}]$ and $[x_{p',j},x_{p',j+1}]$. 
Isotope them so that they lie on $\partial B'$. 
The arcs obtained are the only arcs on $\partial B'$ that intersect once the boundary 
of the compressing disk (see \Cref{fig:compressing disk for tangle}).
The same holds for their image under $\varphi$. This proves (3b).
\end{proof}

\begin{lemma}\label{lem: permutations}
For $K$, $K'$ satisfying the conditions of \Cref{assum:FlipVertical}, up to an 
isotopy, and up to a  possible rotation in a horizontal axis, 
$\varphi(\alpha_{i,j}) = \alpha_{i,j}$ for all $1\le i\le n$, and $1\le j\le r$. 
\end{lemma}

\begin{proof}
\textbf{Step 1.} Up to an isotopy, for any layer $\ell_i$, $1\le i\le n$, there exist indices 
$1\le j(i)\le r$, and $1\le q(i)\le n$ such that $\varphi$ maps the arc $\alpha_{i,j(i)}$ 
in the $j(i)$-th slab to the arc $\alpha_{q(i),j(i)}$ still in the $j(i)$-th slab, as follows:

Let $1\le i\le n$. Consider the twist boxes of $K'$ in the layer $\ell_i$. Each lies in 
some slab. That is, there exist indices $1\le j<r$ such that $i=p'_j$ using the terminology 
of \Cref{def:VerticalSlab}(4). Set $j(i)$ to be the first $j$ for which $i=p'_j$, and set 
$q(i) = p_{j(i)}$. By \Cref{lem: twists between consecutive}(3a), 
$\varphi(\alpha_{i,j(i)})=\alpha_{q(i),j(i)}$.

\vskip7pt

\textbf{Step 2.} 
Up to an isotopy, there exists a permutation $\pi\from\{1,\dots,n\}\to \{1,\dots,n\}$ 
such that for any layer $\ell_i$, $1\le i\le n$ and any slab $W_j$, $1\le j\le r$, 
the map $\varphi$ sends the arc $\alpha_{i,j}$ to the arc $\alpha_{\pi(i),j}$. 
That is, the result holds for all arcs $\alpha_{i,j}$, not just those adjacent to 
twist boxes, and the permutation $\pi(i)$ is independent of the slab $j$:

By the previous step, for all $1\le i \le n$ there exists $1\le j=j(i)\le r$ and 
$1\le q(i)\le n$ such that $\varphi(\alpha_{i,j})=\alpha_{q(i),j}$. 
To prove step 2, we show that $q:=q(i)$ is independent of the second index $j$. 
We first show that if $\varphi(\alpha_{i,j}) = \alpha_{q,j}$ and $j<r$ then 
$\varphi(\alpha_{i,j+1})=\alpha_{q,j+1}$.

Consider the rectangle $Q'$ on $P$ with vertices $x_{i-1,j},\;x_{i,j},\;x_{i,j+1},\; x_{i-1,j+1}$.
It consists of two vertical arcs $\alpha_{i,j},\alpha_{i,j+1}$ and two horizontal arcs that go 
across the slab $W_j$. By assumption the vertical arc $\varphi(\alpha_{i,j})=\alpha_{q,j}$.
Each of the horizontal arcs either lies on $K'$ or is the top or bottom boundary of the unique 
twist box in $W_j$. In both cases, the horizontal arcs of $Q'$ are mapped to the horizontal 
arcs $[x_{q-1,j},x_{q-1,j+1}]$ and $[x_{q,j},x_{q,j+1}]$ across the slab $W_j$ (in the latter 
case, this follows by \Cref{lem: twists between consecutive}(3b)).

The previous paragraph shows that the image $\varphi(Q')$ is made up of the 
vertical arc $\alpha_{q,j}$, the horizontal arcs  $[x_{q-1,j},x_{q-1,j+1}]$ and
$[x_{q,j},x_{q,j+1}]$, and some arc $\alpha := 
\varphi(\alpha_{i,j+1})\subset S_{j+1}$.  Since $Q'$ bounds a disk in $W_j\ssm K'$, 
the image $\varphi(Q')$ bounds a disk in $W_j \ssm K$. It follows that $\alpha$ must 
be isotopic to the vertical arc $\alpha_{q,j+1}$, as this is the only arc with 
this property.

Similarly, one shows that if $j>0$ and $\varphi (\alpha_{i,j}) = \alpha_{q,j}$ 
then $\varphi(\alpha_{i,j-1})=\alpha_{q,j-1}$.  By induction it follows that for 
all $1\le j \le r$, $\varphi(\alpha_{i,j})=\alpha_{q,j}$ for $q=q(i)$.
\vskip7pt
Set $\pi\from \{1,\dots,n\}\to \{1,\dots,n\}$ to be the permutation $\pi(i)=q (= q(i))$.

\vskip7pt

\textbf{Step 3.} Up to an isotopy and a possible 
rotation in a horizontal axis, $\pi = \id$:

The union of the vertical arcs in order $\alpha_{1,j} \cup \dots \cup \alpha_{n,j}$ 
forms a connected arc in the vertical 2-sphere $S_j$. Hence the image 
under $\varphi$ also forms a connected arc comprised of vertical arcs. By the 
previous step, this image is $\alpha_{\pi(1),j} \cup \dots \cup \alpha_{\pi(n),j}$. 
It follows that either $\pi = \id$ or the involution $i\mapsto n-i$.
Therefore after possibly applying a rotation in a 
horizontal axis, we may assume that $\pi=\id$.
\end{proof}

After perhaps applying a rotation in a horizontal axis we may assume 
from now on:
\begin{assumption}\label{Assump:FlipHorizontal}
In addition to \Cref{assum:FlipVertical} assume that $\varphi(\alpha_{i,j})=\alpha_{i,j}$.
\end{assumption}

\begin{lemma}
With \Cref{Assump:FlipHorizontal},
    the part of the diagram of $K'$ between $S_1$ and $S_r$ and the part of 
    the diagram of $K$ in the closed region between $S_1$ and $S_r$ are 
    identical (up to an isotopy of the diagram fixing the vertical spheres and 
    preserving the layers).
\end{lemma}

\begin{proof}
We already assumed that the vertical spheres of $K$ and $K'$ are identical. 
Consider the vertical slab $W_j$. The knot $K$ (resp.\ $K'$) has one twist 
region in $W_j$ at the layer $\ell_{p_j}$ (resp.\ $\ell_{p'_j}$) with $k_j$ 
(resp. $k'_j$) signed crossings. By \Cref{lem: twists between consecutive}, 
$k_j = k'_j$, and \Cref{lem: permutations}, $$\alpha_{p'_j,j}=\varphi(\alpha_{p'_j,j}) = 
\alpha_{p_j,j}.$$ Hence $p_j = p_j'$. This means that the unique twist region of 
$K$ in the slab $W_j$ is in the same layer and has the same number of  crossings 
as that of $K'$.
\end{proof}

\begin{assumption}\label{assumption: same middle part}
Assume in addition to \Cref{Assump:FlipHorizontal} that the diagrams of $K$ and $K'$ 
are identical in the region between  $S_1$ and $S_r$.
\end{assumption}

\begin{lemma}\label{cor: identity in the middle}
With \Cref{assumption: same middle part}, the homeomorphism $\varphi$ restricted 
to the closed region between $S_1$ and $S_r$ is (isotopic to) the identity map.
\end{lemma}

\begin{proof} 
By \Cref{Assump:FlipHorizontal}, $\varphi(\alpha_{i,j}) = \alpha_{i,j}$.
Thus, $\varphi$ fixes the arc $\alpha_{1,j}\cup \dots \cup \alpha_{n,j}$, 
and cutting $S_j$ along this arc yields a disk. It follows that $\varphi$ 
is isotopic to the identity map on $S_j$. Now, as we saw in the 
proof of \Cref{lem: twists between consecutive}, we may assume that $\varphi$ 
preserves the annulus $A\subseteq W_j$ around the unique twist region. 
The homeomorphism $\varphi$ restricted to $B$ induces a self homeomorphism of 
the tangle $(B,B\cap K)$. The only self homeomorphisms of a rational tangle are 
given by rotations along the compressing disk of the tangle. Since the boundary 
of the compressing disk intersects the annulus $A$ essentially, and $\varphi(A) = A$, 
we get that restriction of $\varphi$ to $B$ is isotopic to the identity in $B$ 
relative to $B\cap K$.

Because a vertical slab is a product away from the twist region, $\varphi$ is 
isotopic to the identity on each vertical slab.
\end{proof}

We summarize the results of this section with the following.
\begin{corollary}\label{cor: middle part}
    Let $K,K'$ be two $4$-highly twisted plat projections, 
    on the plane $P$, of widths $m,m'\ge 4$ and heights $n,n'$ respectively. 
    Assume that $K$ and $K'$ represent the same knot $\mathcal{K}$, and 
    let $\varphi\from (S^3,K')\to (S^3,K)$ be an ambient isotopy taking $K'$ to $K$. 
    Then, up to a horizontal and/or vertical rotation of the diagram $K$, an 
    isotopy of the diagram, and an ambient isotopy relative to $K$, we have:
    \begin{enumerate}
        \item $m=m'$, $n=n'$.
        \item There are spheres $S_L,S_R$ which are respectively the leftmost 
        and rightmost vertical spheres for both $K,K'$. 
        \item The isotopy $\varphi$ restricts to the identity on the region 
        between $S_L$ and $S_R$, and in particular the diagrams of $K$ and $K'$ 
        are identical there.
    \end{enumerate}
\end{corollary}

\section{2-bridge knots and links}\label{sec:2bridge}

In this section, we focus on the leftmost part of the plat, but all the arguments 
in this section work verbatim for the rightmost part of the plat. 
By \Cref{cor: middle part} the homeomorphism $\varphi$ of 
$(S^3, \mathcal{K} )$ taking the plat diagram $K'$ to $K$ 
can be assumed to be the identity on the leftmost vertical sphere $S_L$.
Consider the ball $B_L$ in $(S^3, K')$ bounded on the left of $S_L$, and 
the tangle defined by the arcs $B_L\cap K'$. 

Define a 2-bridge knot or link $K_L'$ by completing the resulting 
tangle into a 2-bridge knot or link as follows: Connect the topmost puncture on 
$S_L$ to the bottommost with a trivial arc in the ball $S^3 \ssm B_L$ 
on the right of $S_L$, and then connect each pair of adjacent punctures, again 
by a trivial arc to the right of $S_L$. This procedure results in the Schubert 
normal form of a 2-bridge knot or link, as in Figure~\ref{fig:2bridge}.

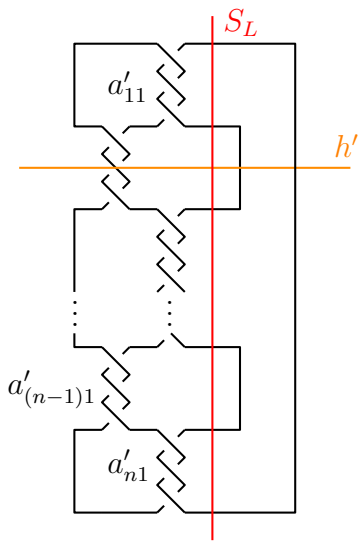
\begin{figure}[!ht]
    \centering
    \begin{overpic}[height=7cm]{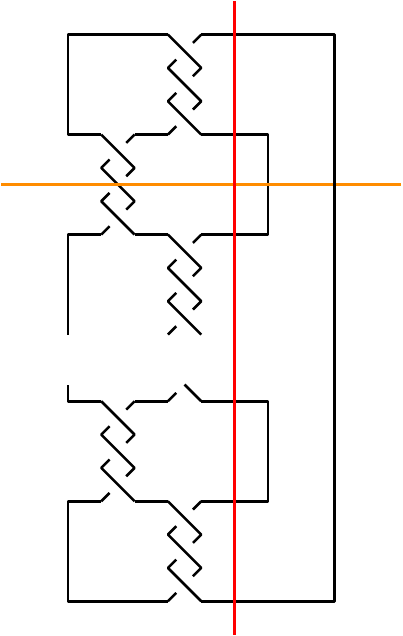}
        \put(17,85){$a'_{11}$}
        \put(-2,28){$a'_{(n-1)1}$}
        \put(17,13){$a'_{n1}$}
        \put(28,41){$\vdots$}
        \put(10,41){$\vdots$}
        \put(39,97){\color{red}$S_L$}
        \put(60,73){\color{orange}$h'$}
    \end{overpic}
    \caption{The completion of the left portion of the plat into a $2$-bridge knot or link.}
    \label{fig:2bridge}
\end{figure}

Perform the same operation on the corresponding tangle in $(S^3,K)$. 
One obtains another $2$-bridge knot or link $K_L$. 
We call $K'_L$ and  $K_L$ the left $2$-bridge knots or links. 

Let $\varphi':(S^3,K_L') \to (S^3,K_L)$ be the homeomorphism that is equal 
to $\varphi$ on $B_L$ and is the identity outside $B_L$.

\subsection{Continued fractions and 2-bridge knots and links.}
For integers $a_0,\dots,a_n$, let $[a_0;a_1,\dots,a_n]$ denote the continued fraction
expansion of the rational number $r$:
\[ r = a_0 + \frac{1}{a_1 + \frac{1}{\ddots\;+\frac1{a_n}}} \]

\begin{lemma}\label{uniqueness of continued fractions}
Let $[a_0;a_1,a_2,\dots,a_n] = [b_0;b_1,\dots,b_m]$ be two continued 
fractions expansions of the same rational number $r$, and assume $|a_i|\ge 3$ 
and $|b_j|\ge 3$ for all $i,j$. Then, $n=m$ and $a_i = b_i$ for all $1\le i\le n$.
\end{lemma}

\begin{proof}
First, we observe that if $|a_i|\ge 3$ then 
\begin{equation}\label{upperbound on cf}
|[0;a_1,\dots,a_n]| \le [0;3,-3,3,-3,\dots] = \tfrac{3 - \sqrt{5}}2<\tfrac12.
\end{equation}
    
We prove the lemma by induction on $\max\{m,n\}$. If $\max\{m,n\}=0$ then $r=a_0=b_0$. 
Now if $\max\{m,n\}>0$, since $r=a_0 + [0;a_1,\dots,a_n]$, it follows from 
\eqref{upperbound on cf} that $a_0=b_0$ since they are the (unique) closest integer 
to $r$. Now $$[a_1;a_2,\dots,a_n] =\tfrac{1}{r-a_0} = [b_1;b_2,\dots,b_m],$$ and 
by the induction hypothesis $m=n$ and $a_i=b_i$ for all $1\le i \le n$.
\end{proof}

\begin{corollary}\label{cor: unique}
The 3-highly twisted Schubert normal form of a 2-bridge knot is unique up 
to a rotation around the horizontal axis.
\end{corollary}

\begin{proof}
Assume the Schubert's normal form of a 2-bridge knot has twist regions  
with corresponding coefficients  $a_1,\dots,a_n$.  

Let  $r=[a_1, -a_2,\dots,-a_{n-1}, a_n]$ 
and $r'=[a_n, -a_{n-1},\dots,-a_2,a_1]$ be the corresponding rational numbers. 
By Schubert~\cite{Schu} the unordered pair $\{r,r'\}$ is a (complete) 2-bridge knot
invariant. 
\Cref{uniqueness of continued fractions} implies
the numbers $a_1,\dots,a_n$ are determined by $r$, and $a_n,\dots,a_1$ are 
determined by $r'$. These expansions correspond to a rotation around the 
horizontal axis.
\end{proof}

Going back to our plats $K',K$, we have produced two isotopic 2-bridge knots 
$K_L',K_L$ which are 3-highly twisted (in fact 4-highly twisted) and in Schubert 
normal form. Hence we have the following corollary.

\begin{corollary}\label{Cor:Coeffs2Bridge}
Either $a_{i1}' = a_{i1}$ for $i=1, \dots, n$, or $a_{i1}' = a_{n-i, 1}$ for
$i=1, \dots, n$.  \qed
\end{corollary}

By \Cref{Cor:Coeffs2Bridge}, the isotopy $\varphi'\from (S^3,K'_L)\to (S^3,K_L)$ 
takes $K'_L$ to an isotopic $2$-bridge knot, isotopic to $K_L$, where the isotopy 
will either identify twist regions, or rotate the knot through a horizontal axis.
Because we already adjusted $K$ by rotating along a horizontal axis in
\Cref{Assump:FlipHorizontal}, we need to rule out any additional rotation. 
To do so, we will use the fact that $\varphi'$ fixes the sphere $S_L$ and 
the ball $S^3 \ssm B_L$ bounded by it. Any isotopy $\psi'\from  (S^3, \varphi'(K'_L))\to
(S^3, K_L)$ that rotates along a horizontal axis, inverting the diagram of
$\varphi'(K'_L)$, cannot restrict to the identity in $S^3\ssm B_L$. Thus to rule 
out inverting the coefficients, we show that there is an isotopy taking 
$\varphi'(K_L)$ to $K_L$ that is the identity on $S^3\ssm B_L$.

Consider the horizontal plane $h'$ (resp.\ $h$) in $S^3$ that is perpendicular 
to $P$ and crosses the knot $K'_L$ (resp.\ $K_L$) in the middle of
the second twist box, as in \Cref{fig:2bridge}.

\begin{lemma}\label{Lem:HorizontalBridgeSphere2Bridge}
Up to postcomposing $\varphi'$ with an isotopy supported in $B_L$, 
we have  $\varphi'(h')=h$.
% , and $\varphi'$ is the identity map restricted to $h' \to h$. 
\end{lemma}

\begin{proof}
Let $b$ denote $\varphi'(h')$. 
Schubert~\cite{Schu} showed that all $2$-bridge spheres for a $2$-bridge 
knot or link are isotopic, implying that $b$ is isotopic to $h$.  
Here, we upgrade the proof of Schubert's theorem, to show
that in our case $\varphi'$ is in fact isotopic to the 
identity on $h'$.  Note that we already have $h=h'$ in $S^3\ssm B_L$. 
We will show that the isotopy $\varphi$ can be taken relative to $\bdy B_L = S_L$
so it fixes $S^3 \ssm B_L$ as well. 

We may assume that $b$ and $h$ intersect  transversely in $B_L$. Since 
$b\cap \bdy B_L$ is a component of $b\cap h$, all other intersections must 
be simple closed curves. An innermost component of $b\cap h$ on $h$ bounds 
a disk $D_h$ in $h\cap B_L$ and a disk $D_b$ in $b$. The union $D_h \cup D_b$
is a $2$-sphere in $B_L$. Note that $K_L$ can meet $D_h$ (and $D_b$) in exactly 
$0$, $1$, or $2$ points. 

\vskip7pt

\textbf{Step 1. } Suppose first that $K_L$ is disjoint from $D_h$ and $D_b$. As $
S^3\ssm K_L$ is irreducible, the sphere $D_h\cup D_b$ bounds a $3$-ball that 
is disjoint from $K_L$, and thus disjoint from $S_L$. Isotope $\varphi$ through 
this $3$-ball, while keeping $S^3 \ssm B_L$ fixed, to remove this intersection 
curve $\partial D_h = \partial D_b$. After iterating this procedure finitely 
many times, we may assume that there are no intersection curves of $h$ and $b$ 
that bound disks  disjoint from $K_L$ on the side disjoint from $S_L$.

\vskip7pt

\textbf{Step 2.} Suppose that $D_h$ meets $K_L$ exactly twice, and that $D_b$ is disjoint 
from $K_L$. We may assume that the interior of $D_b$ is disjoint from $h$, as follows:
We have removed all curves of intersection in $b\cap h$ bounding disks disjoint 
from $K_L$, so any intersections $D_b\cap h$ are parallel to $\bdy D_h$ in $h$. Then 
replace $D_b$ and $D_h$ with an innermost subdisk of $D_b$ and a larger subdisk of $h$ 
with the same boundary. So $D_b$ is on exactly one side of $h$, either above or below it,
disjoint from $K_L$, and $D_h$ meets $K_L$ exactly twice. 

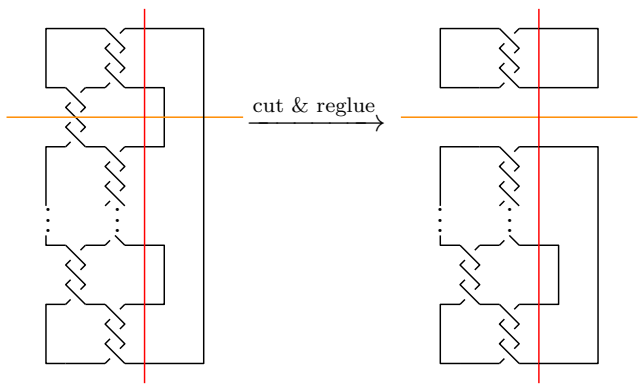
\begin{figure}[!ht]
    \centering
    \begin{overpic}[height=5cm]{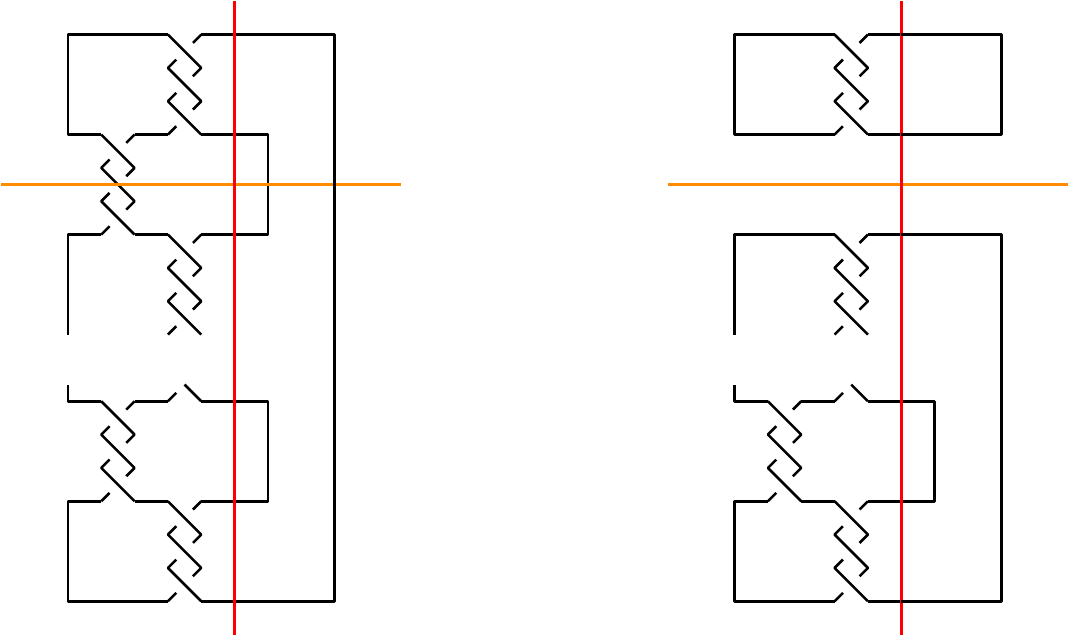}
        \put(6,24){$\vdots$}
        \put(17,24){$\vdots$}
        \put(68,24){$\vdots$}
        \put(79,24){$\vdots$}
        \put(38,40){$\xrightarrow{\text{cut \& reglue}}$}
    \end{overpic}
    \caption{Cut  the 2-bridge along $h$ and reglue into a torus knot/link 
    and a shorter 2-bridge knot.}
    \label{fig:2bridgeSurgery}
\end{figure}

A small regular neighborhood of  the disk $D_h$ contains two small arcs of $K_L$
running vertically from the ``bottom'' of the neighborhood to the ``top'' Thus it defines a 
trivial tangle that is a ball with two simple unknotted vertical strands running through it. 
Replace this trivial tangle by reconnecting the endpoints of the vertical strands: i.e.,  
by connecting them by two horizontal strands rather than vertical strands. Alternatively, 
this can be described as replacing the $1/0$ tangle formed by $(N(D_h),K_L\cap N(D_h))$ 
with a $0/1$ tangle in $N(D_h)$. Similarly, replace a thickened $N((S^3\ssm B_L)\cap h)$ 
by a tangle in which strands connect horizontally rather than vertically. This changes 
the diagram of the $2$-bridge knot/link into two diagrams of two $2$-bridge knots/links, 
one above $h$ and one below, as in \Cref{fig:2bridgeSurgery}. By choice of $h$, each 
of these diagrams contains at least one twist region as $n$ is odd and greater or equal 
to three. Each diagram is still in Schubert normal form, and remains $3$-highly twisted. 

Through the center of the neighborhood of $D_h$ there is now a new  horizontal disk, also 
denoted by $D_h$, that is disjoint from both new diagrams. The disk $D_b$ remains disjoint 
from the diagrams, and it is either above or below $h$. Consider the new $2$-bridge 
knot/link  lying on the same side of $h$ as $D_b$. The union of the two disks 
$D_b\cup D_h'$  forms an embedded $2$-sphere in the complement of the $2$-bridge 
knot/link determined  by the diagram lying on the side of $h$  containing $D_b$. 
This $2$-sphere cannot bound a ball, because it contains points of $K_L$ 
meeting the original disk $D_h$ on one side, and points of $K_L$ outside of $B_L$
on the other side. This is a contradiction: If the diagram has only one twist 
region, it is a torus knot/link, which is irreducible. If it has more than one 
twist region, it is hyperbolic, for example by ~\cite{LMP}, and thus is 
also irreducible.

\vskip7pt

\textbf{Step 3.} Suppose next that $K_L$ meets $D_h$ in exactly one point. Then it 
must also meet $D_b$ in exactly one point. If $D_b$ is not innermost, then it must 
contain curves of intersection of $D_b\cap h$ that encircle the single 
point of intersection $D_b\cap K_L$, and these curves again must encircle 
a single point of intersection of $K_L\cap h$. Thus we may replace $D_b$ by an
innermost curve on $D_b$, replacing $D_h$ by a disk with the same boundary. 
It follows that $D_b$ and $D_h$ will both meet $K_L$ exactly once, and the
interior of $D_b$ will lie either fully above $h$ or fully below $h$. 
Because $h$ is a bridge sphere, there is a homeomorphism $\psi$ taking the 
region above $h$ to a trivial 
tangle. The annulus $D_h \cup D_b \ssm K_L$ is sent to a meridional annulus 
$\psi(D_h\cup D_b) \ssm \psi(K_L)$ in the trivial tangle. Thus it must be boundary 
parallel. Using this annulus, we can isotope $D_b$ through the trivial tangle 
to remove the intersection curve $\alpha$. Similarly if $D_b$ lies below, 
we use $\psi'$ to isotope away the intersection curve $\alpha$. 
Repeating this procedure finitely many times, we may assume neither $D_h$ 
nor $D_b$ intersect $K_L$ exactly once. 

\vskip7pt

\textbf{Step 4.} The only remaining possibility is that $K_L$ meets $D_h$ exactly twice and 
$D_b$ exactly twice. By the previous steps, all components of $b\cap h$ in $B_L\cap h$ are 
parallel. There must be a pair of curves in $b\cap h$ which bound annuli 
$A_h\subset h$ and $A_b\subset b$ that are both innermost in $b$ and in $h$. The union 
$A_h\cup A_b$ is a torus in $S^3 \ssm K_L$. Since there are no essential tori in the 
complement of a $2$-bridge knot (see ~\cite{HatcherThurston1985Incompressible}), it must 
bound a solid torus that is disjoint from $K_L$. Use this solid torus to isotope $A_b$ 
through $A_h$, removing a pair of simple closed curves of $b\cap h$. Note that since the 
solid torus is disjoint from $K_L$ it is also disjoint from $S_L$, and so this isotopy 
fixes $S_L$. Repeating this a finite number of times, we remove all intersection 
components of $b\cap h$ in the interior of $B_L$. Thus $b\cap h$ consists of 
$b\cap h\cap S_L$ and the disk $h\cap S^3\ssm B_L$ punctured twice by $K_L$.

It remains to show that $D_h$ and $D_b$ are parallel in $B_L$ rel $S_L$. 
Let us consider $D_b$ as a properly embedded surface in the ball $A$ above $h$.
As in Step 2, $D_b \ssm  K$ is incompressible in $A \ssm K$.  Hence, 
by \cite[Lemma 3.6]{ScharlemannTomova2008Uniqueness}, $D_b \ssm K$ is 
$\partial$-compressible in $A\ssm K$. Let us emphasize that we remove $K$ and 
not $\NN(K)$ here, so the the boundary of $A\ssm K$ is the four punctured sphere 
$\partial A \ssm  K = h \ssm K$. Let $E$ be a boundary compression disk for 
$D_b \ssm  K$. Its boundary $\partial E$ is the union  $\gamma_h \cup \gamma_b$ 
of an arc $\gamma_b \subset D_b \ssm K$ and an arc 
$\gamma_h \subset \partial A \ssm K = h \ssm K$.
The arc $\gamma_b$ must be a non-trivial arc in $D_b \ssm K$, namely, an arc 
that separates the two punctures. The arc $\gamma_h$ is an arc on $h\ssm K$ 
whose endpoints are on $\partial D_b = \partial D_h = h\cap S_L$. There are two 
possibilities: $\gamma_h \subset D_h$ or $\gamma_h \subset h \ssm \overset{\circ}{D_h} $.

Let us first show that the latter case does not occur. Indeed, if 
$\gamma_h\subset h \ssm \overset{\circ}{D_h} =  b\ssm \overset{\circ}{D_b}$, 
then the disk $E$ is a compression disk for $b \ssm K$ in the trivial tangle 
above $b$. Its preimage $E':=\varphi'^{-1}(E)$ is a compression disk for $h'$ 
in the trivial tangle above it. In a trivial 2-tangle there is a unique 
compression disk. Since there is a twist region above $h'$ in $K'$, the 
boundary $\partial E'$ must intersect the circle $h'\cap S_L$ in more than 
two points (in fact, they intersect at least $2|a_{1,1}'|$ times). Therefore, its 
image under $\varphi'$, $\partial E$, intersects $\varphi'(h'\cap S_L)=h\cap S_L$ 
in more than two points. However, this contradicts the fact that the intersection 
$\partial E\cap S_L$ is exactly the two endpoints of $\gamma_b$.

Finally, if $\gamma_h \subset D_h$, then the disk $E$ is contained in $B_L$. 
The disk $E$ separates the ball bounded by $D_h \cup D_b$ into two balls. 
The knot $K$ meets the boundary of each of these balls in two points -- one 
on $D_h$ and the other on $D_b$. Since $K$ is prime, in each of these balls 
there is a trivial 1-tangle. The sphere $D_h \cup D_b$ bounds a trivial $2$-tangle 
whose stands run across between $D_h$ and $D_b$. Therefore, we can isotope $D_b$ 
to coincide with $D_h$. Note that this isotopy is supported in $B_L$. Thus we 
have shown that there is an isotopy supported in $B_L$ that sends $b$ to $h$. 
By postcomposing $\varphi'$ with that isotopy, we get $\varphi'(h')=h$.
\end{proof}

\begin{proposition}\label{Prop:Coeffs2Bridge}
The coefficients $a_{11}',\dots,a_{n1}'$ of $K_L'$ and $a_{11},\dots,a_{n1}$ 
of $K_L$ satisfy $a'_{i1}=a_{i1}$  for all $i=1,\dots,n$. Moreover,
$\varphi'$ is isotopic to the identity. 
\end{proposition}

\begin{proof}
By \Cref{Lem:HorizontalBridgeSphere2Bridge}, the isotopy $\varphi'\from 
(S^3,K_L')\to (S^3,K_L)$ can be taken to be the identity on all 
of $S^3\ssm B_L$, and on the 2-bridge sphere $h'=h$. Thus it also 
must take the ball above $h'=h$ to the ball above $h'=h$, and similarly
for the ball below $h$ below. 

Kobayashi and Morimoto-Sakuma (\cite{Kobayashi}, \cite{MorimotoSakuma}) classified the 
tunnels of $2$-bridge knots and showed that a 2-bridge knot has at most six 
unknotting tunnels. Lemma 5.1 and Theorem 5.2 of \cite{MorimotoSakuma} show that
the tunnels $\tau_1$ connecting  the two maxima horizontally, and  $\tau_2$ connecting
the two minima horizontally are unique  (up to isotopy). It follows, from work of Bleiler 
and Moriah~\cite{BleilerMoriah} for the case of a knot and from work of Gueritaud 
\cite{Gueritaud2006Canonical} for the case  of a link, that an isotopy taking a 2-bridge
knot/link to itself preserves $\tau_1\cup\tau_2$ up to isotopy in $S^3\ssm K_L'$, and 
is determined up to isotopy according to whether it switches the two tunnels, and whether 
it switches their orientations.

Since $\tau_1$ can be isotoped into $h$ and since by Lemma \ref {Lem:HorizontalBridgeSphere2Bridge}
$\varphi'$ can be  assumed to be the identity on $h$,  $\varphi'$ fixes pointwise the
upper tunnel $\tau_1$ with its orientation, up to isotopy in $S^3\ssm K_L'$. A similar
argument using a horizontal sphere going through the second from the bottom twist region  
shows that $\varphi'$ also fixes the lower tunnel $\tau_2$ with its orientation, up to
isotopy in $S^3\ssm K_L'$. Thus, by \cite{BleilerMoriah, Gueritaud2006Canonical}, $\varphi'$ must be
isotopic to the identity map and therefore $a_{i1} '= a_{i1} $ for all $1\le i \le n$ as claimed.
\end{proof}

\vskip15pt

\section{Proof of the main result, applications to braids}
\label{sec:Applications}

\vskip10pt

\subsection{Proof of Theorem \ref{thm:UniquePlat}}
We restate Theorem~\ref{thm:UniquePlat} for convenience.

\begin{named}{\Cref{thm:UniquePlat}}
Let $K'$ and $K$ be two $4$-highly twisted plats 
representing the same knot or link $\mathcal{K} \subset S^3$, so that 
each diagram has width greater than or equal to $4$, and odd height greater 
than or  equal to $3$. Then $K=K'$ up to rotation in a 
vertical and/or a horizontal axis of the plats.
\end{named}

\begin{proof} 
By \Cref{cor: middle part}, the knots $K$ and $K'$ have the same height and width. 
Up to rotation in a horizontal and/or vertical axis, the region between the leftmost 
and rightmost vertical spheres $S_L$ and $S_R$ in $(S^3,K')$ is mapped by the identity
to the region between $S_L$ and $S_R$ in $(S^3,K)$. Thus up to these rotations, 
it follows that twist regions between these spheres have identical locations and 
identical number of crossings. 

By \Cref{Prop:Coeffs2Bridge}, the twist regions to the left of $S_L$ are mapped to 
twist regions to the left of $S_L$ with the same numbers of crossings, by an isotopy 
that is isotopic to the identity. An identical argument gives the same result for the 
twist regions to the right of $S_R$ and $S_R$. 

Thus the diagrams of $K'$ and $K$ are identical.
\end{proof}

\begin{corollary}\label{cor:Symmetries}
The symmetry group (consisting of self homeomorphisms up to isotopies) of  $S^3\ssm K$, 
where $K$ is a knot or link with a $4$-highly twisted plat diagram, contains at most four elements. 
\end{corollary}

\begin{proof}
    It follows from the proof of Theorem \ref{thm:UniquePlat} that any self homeomorphism of 
    $S^3\ssm K$ followed by an isotopy of $S^3\ssm K$ can be assumed to take the set of 
    vertical spheres to itself, and a bridge sphere to itself. By post-composing with a 
    $\pi$ rotation around a horizontal axis through the middle of the plat, a vertical 
    rotation through the middle of the plat, or a composition of both (which is a 
    rotation around an axis perpendicular to the center point of the plat), one may 
    assume the spheres are preserved with their orientations, at which point it 
    follows from \Cref{cor: identity in the middle} and the proof of 
    \Cref{Prop:Coeffs2Bridge} that the homeomorphism is isotopic to the identity.
\end{proof}

It seems likely, that when the array of twist numbers has extra symmetry 
(i.e., is palindromic or has a symmetry with respect to a horizontal rotation, 
or both, the corresponding elements in the group are non trivial (that is, 
these elements are not isotopic to the identity). It is interesting to 
compare to the case of 2-bridge knots and links \cite{MillichapWorden}.

\vskip10pt

This paper focuses on highly twisted plat diagrams, but there are two 
other natural ``plat-like'' closures of a braid -- \emph{even plats} 
and \emph{doubly even plats} -- which we define next.

\begin{definition}\label{def: even and doubly even}
Let $b$ be a word in $\mathcal{X}_m$. Consider the corresponding braid diagram, 
and let $x_1,\dots,x_{2m}$, be the top endpoints, and $y_1,\dots,y_{2m}$ 
be the bottom endpoints of the strands (as in \Cref{def: plat closure}). The 
\emph{even plat closure} of $b$ is obtained  by connecting the pairs 
$\{x_1,x_2\},\dots,\{x_{2m-1},x_{2m}\}$ at the top by small unknotted disjoint 
arcs, identically to the case of the plat closure, but connecting 
the pairs $\{y_2,y_3\},\dots,\{y_{2m},y_1\}$ at the bottom by small unknotted 
disjoint arcs; see Figure~\ref{fig:even_plat}. 

\begin{figure}[ht]
    \centering
    \includegraphics[width=0.5\linewidth]{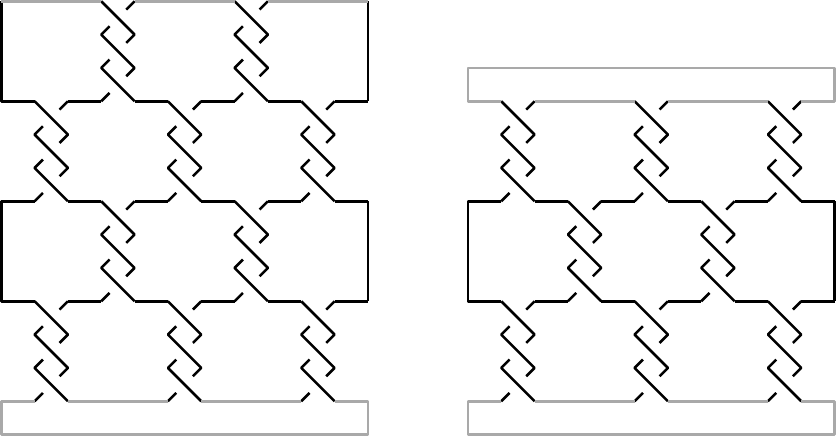}
    \caption{Even and doubly even plat closures.}
    \label{fig:even_plat}
\end{figure}

Similarly, the \emph{doubly even plat closure} of $b$ is obtained by connecting 
the pairs  $\{x_2,x_3\},\dots,\{x_{2m},x_1\}$ at the top and  
$\{y_2,y_3\},\dots,\{y_{2m},y_1\}$ at the bottom.
\end{definition}

One can similarly define the notions of \emph{standard form}, and \emph{highly twisted} 
for even and doubly even plats.

\begin{conjecture}\label{Con: Also even plats}
The claims of  \Cref{thm:UniquePlat} hold for such diagrams as well.
\end{conjecture}

\subsection{Braids}\label{sec: braids}

\Cref{thm:UniquePlat} has applications to the form of braids in the Hilden subgroup, 
which we now review. 

\begin{definition}[Hilden \cite{Hilden}]\label{def: Hilden subgroup }
The \emph{Hilden subgroup} is the subgroup of the braid group $\mathcal{B}_{2m}$ 
generated by the following elements, called \emph{Hilden moves}:

\begin{enumerate}

\item  For $i$ odd let $\bf{h}^1_i = \sigma_i$.

\item  For $i$ odd $i \neq 2m -1$ let $\bf{h}^2_i = 
\sigma_{i+1}\cdot\sigma_{i+2}\cdot\sigma_i\cdot\sigma_{i+1}$.

\item For $i$ odd $i \neq 2m -1$ let 
$\mathbf{h_i^3} = \sigma_{i+1}\cdot \sigma_i \cdot \sigma_{i+2}^{-1} \cdot \sigma_{i+1}^{-1}$.

\item For $i$ odd $i \neq 2m -1$ let 
$\mathbf{h_i^4} = \sigma_{i+1}^{-1}\cdot \sigma_i^{-1} \cdot \sigma_{i+2} \cdot \sigma_{i+1}$.

\end{enumerate} 
\end{definition}

The Hilden moves are depicted in \Cref{fig:hilden moves}. 

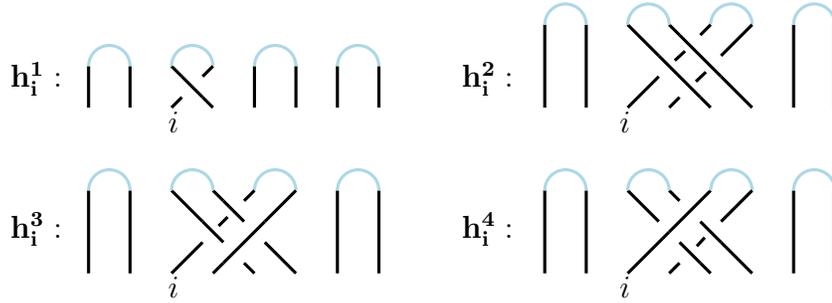
\begin{figure}[ht]
    \centering
    \vspace{0.3cm}
    \begin{overpic}[width=10cm]{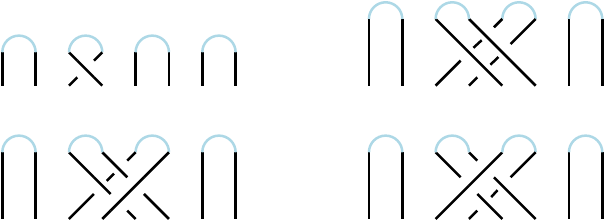}
        \put(-10,25){$\mathbf{h_i^1}:$}
        \put(11,19){$i$}
        \put(50,25){$\mathbf{h_i^2}:$}
        \put(71,19){$i$}
        \put(-10,5){$\mathbf{h_i^3}:$}
        \put(11,-3){$i$}
        \put(50,5){$\mathbf{h_i^4}:$}
        \put(71,-3){$i$}
    \end{overpic}
    \vspace{0.5cm}
    
    \caption{The Hilden moves acting on a standard set of four bridges.}
    \label{fig:hilden moves}
\end{figure}

\vskip 20pt

It is evident from \Cref{fig:hilden moves} (see also the discussion on page 476 in
\cite{Hilden}), that multiplying a braid $b$ on the right and on the left by elements 
$h_1, h_2 \in \mathcal {H}$ does not change the isotopy type of its plat closure. That is, 
plat closures of $b$ and $h_1 b h_2$ are diagrams of the same knot or link. 
Using this observation,
\Cref{thm:UniquePlat} on the uniqueness of the plat projection of a $4$-highly twisted plat
can be translated into uniqueness of $4$-highly twisted representatives of double cosets in
$\mathcal{H}\backslash \mathcal{B}_{2m} / \mathcal{H}$. 

\begin{corollary}\label{cor:HildenDoubleCosets}
If $b$ and $b'$ are two $4$-highly twisted words in $\mathcal{X}_{2m}$ with the same 
Hilden double coset $\mathcal{H}b\mathcal{H} = \mathcal{H}b'\mathcal{H}$, then $b$ 
and $b'$ are the same up to a $\pi$ rotation of their corresponding braid diagrams
around a vertical and/or a horizontal axis.  \qedsymbol
\end{corollary}

\bibliographystyle{amsplain}
\bibliography{references.bib}

\end{document}